\documentclass{amsart}
\usepackage{amsmath}
\usepackage{amsthm}
\usepackage{graphicx}
\usepackage{amsfonts}
\usepackage{amssymb}
\usepackage{textcomp}
\usepackage{color}
\usepackage{mathrsfs}
\usepackage{epsfig}
\usepackage{url}
\usepackage{mathrsfs,a4wide}
\usepackage{calligra}
\usepackage[normalem]{ulem}

\theoremstyle{plain}
\newtheorem*{theorem*}{Theorem}
\newtheorem{theorem}{Theorem}[section]

\newtheorem{lemma}[theorem]{Lemma}
\newtheorem{proposition}[theorem]{Proposition}
\newtheorem{corollary}[theorem]{Corollary}
\newtheorem{remark}[theorem]{Remark}
\newtheorem{definition}[theorem]{Definition}
\theoremstyle{definition}
\theoremstyle{remark}
\numberwithin{equation}{section}

\newcommand{\Ad}{\mathscr{A}}
\newcommand{\mad}{\mathrm{D}}

\newcommand{\ep}{\varepsilon}

\newcommand{\ffi}{\varphi}

\newcommand{\C}{\mathbb{C}}
\newcommand{\R}{\mathbb{R}}
\newcommand{\N}{\mathbb{N}}
\newcommand{\Z}{\mathbb{Z}}

\newcommand{\F}{\mathcal{F}}

\newcommand{\di}{\textrm{dist}}

\newcommand{\ud}{\;\mathrm{d}}

\newcommand{\Om}{\Omega}
\newcommand{\supp}{\mathrm{supp}\,}
\newcommand{\M}{\mathcal{M}}

\newcommand{\Huno}{\mathcal H^1}

\newcommand{\weakly}{\rightharpoonup}           
\newcommand{\weakstar}{\stackrel{*}{\weakly}}   
\newcommand{\fla}{\stackrel{\mathrm{flat}}{\rightarrow}}
\newcommand{\flt}{\mathrm{flat}}

\newcommand{\rad}{\mathcal Rad}

\newcommand{\Cc}{C_{\mathrm{c}}}

\newcommand{\Ss}{\mathbb{S}}

\newcommand{\strictly}{\overset{\mathrm{strict}}\weakly}
\newcommand{\Ccal}{{\mathscr{C}}}
\newcommand{\B}{\mathscr{B}}
\newcommand{\nulla}{\Ccal_\ep^{0}}
\newcommand{\nonnulla}{\Ccal_\ep^{\neq 0}}
\newcommand{\mino}{\Ccal_\ep^{<}}
\newcommand{\magg}{\Ccal_\ep^{>}}
\newcommand{\ite}{\mathscr{S}_\ep}
\newcommand{\itezero}{\mathscr{I}_\ep}
\newcommand{\iteuno}{\mathscr{I}_\ep^<}
\newcommand{\itedue}{\mathscr{I}_\ep^>}
\newcommand{\hatmino}{\widehat{\Ccal}_\ep^{<}}
\newcommand{\hatmagg}{\widehat{\Ccal}_\ep^{>}}
\newcommand{\hatitezero}{\widehat{\mathscr{I}}_\ep}
\newcommand{\hatiteuno}{\widehat{\mathscr{I}}_\ep^{<}}
\newcommand{\hatitedue}{\widehat{\mathscr{I}}_\ep^{>}}
\newcommand{\hatite}{\widehat{\mathscr{S}}_\ep}

\def\XXint#1#2#3{{\setbox0=\hbox{$#1{#2#3}{\int}$}
     \vcenter{\hbox{$#2#3$}}\kern-.5\wd0}}

\newcommand{\res}{\mathop{\hbox{\vrule height 7pt width .5pt depth 0pt
\vrule height .5pt width 6pt depth 0pt}}\nolimits}
\usepackage{latexsym}
\newcommand{\newatop}{\genfrac{}{}{0pt}{1}}

\makeatletter
\def\@splitop#1#2\@nil{$\mathscr{#1}\!\!$\calligra#2\,\,}
\newcommand*\DeclareCursiveOperator[2]{%

 \newcommand#1{\mathop{\mbox{\@splitop#2\@nil}}\nolimits}}
\makeatother
\DeclareCursiveOperator{\Anew}{A}
\DeclareCursiveOperator{\Bnew}{B}
\DeclareCursiveOperator{\Cnew}{C}
\DeclareCursiveOperator{\Dnew}{D}
\DeclareCursiveOperator{\Enew}{E}
\DeclareCursiveOperator{\Qnew}{Q}

\title[Jacobian for $BV$ maps]
{ A new approach to topological singularities via a weak notion of Jacobian for functions of bounded variation}

\author[L. De Luca]
{L. De Luca}
\address[Lucia De Luca]{Istituto per le Applicazioni del Calcolo ``M. Picone'', IAC-CNR, 00185 Rome, Italy}
\email[L. De Luca]{lucia.deluca@cnr.it}
\author[R. Scala]
{R. Scala}
\address[Riccardo Scala]{Dipartimento di Ingegneria dell'Informazione e Scienze Matematiche, Universit\`a di Siena, 53100 Siena, Italy.}
\email[R. Scala]{riccardo.scala@unisi.it}
\author[N. Van Goethem]
{N. Van Goethem}
\address[Nicolas Van Goethem]{Centro de Matem\`atica e Aplica\c{c}\~oes Fundamentais, Universidade de Lisboa, 1749-016, Lisboa, Portugal.}
\email[N. Van Goethem]{vangoeth@fc.ul.pt}

\usepackage{hyperref}
\begin{document}
\begin{abstract}
We introduce a weak  notion of $2\times 2$-minors of gradients of a suitable subclass of $BV$ functions. In the case of maps in $BV(\R^2;\R^2)$ such a notion extends the standard definition of Jacobian determinant to non-Sobolev maps.

We use this distributional Jacobian to prove a compactness and $\Gamma$-convergence result for a new model describing the emergence of topological singularities in two dimensions, in the spirit of Ginzburg-Landau and core-radius approaches. Within our framework, the order parameter is an $SBV$ map $u$ taking values in $\Ss^1$ and the energy is made by the sum of the squared $L^2$ norm of $\nabla u$ and of the length of (the closure of) the jump set of $u$ multiplied by $\frac 1 \ep$\,. Here, $\ep$ is a length-scale parameter.
We show that, in the $|\log\ep|$ regime, the Jacobian distributions converge, as $\ep\to 0^+$\,, to a finite sum $\mu$  of Dirac deltas  with weights multiple of $\pi$, and that the corresponding effective energy is given by the total variation of $\mu$\,.

\end{abstract}
\maketitle
\tableofcontents
\section*{Introduction}
Topological singularities are ubiquitous in Physics and Materials Science: Vortices in superconductivity and superfluidity and (screw end edge) dislocations in single crystal  plasticity are two paradigmatic examples \cite{AO, HL, HB, Lo1, Lo2, Mermin}.
Loosely speaking, dislocations (resp.,  vortices) are identified as points around which the deformation gradient (resp., the order parameter) has non trivial circulation. 
In view of their central role in the theory of phase transitions, topological singularities  have been the object of an extensive and intensive study in the last decades (see e.g.  \cite{BBH,SS2}).
Such an analysis focused basically on two main models: The core-radius approach (CR) and the Ginzburg-Landau model (GL). 
In the former, a finite distribution of topological singularities is identified by a sum of Dirac deltas (with integer weights) and the energy is assumed to be purely elastic ``far enough'' from such singularities. Hence, the {\it core parameter} $\ep$ is introduced, and the energy takes the form
\begin{equation}\label{def:cra_intro}
\mathcal{E}^{\mathrm{CR}}_{\ep}(\mu,u):=\frac{1}{2}\int_{\Omega_\ep(\mu)}|\nabla u|^2\ud x+|\mu|(\Omega)\,.
\end{equation}
In \eqref{def:cra_intro}, $\mu$ represents the distribution of singularities, $\Omega$ is the reference domain, and $\Omega_\ep(\mu):=\Omega\setminus\bigcup_{\xi\in \supp\mu}\overline{B}_\ep(\xi)$ is the domain deprived of the cores; moreover, $u\in H^1(\Omega_\ep(\mu);\Ss^1)$ represents the order parameter that should be ``compatible'' with the measure $\mu$\,, namely, it should satisfy $\deg(u,\partial A)=\mu(A)$ for every smooth domain $A\subset\Omega$ with {$\partial A\subset \Omega_\ep(\mu)$}\,. 
(CR) is mainly used to describe dislocations in semi-discrete models, i.e., models in which ``far from the singularities'' (in $\Omega_\ep(\mu)$) the material is assumed to be continuous and elastic, whereas ``close to the singularities'' (in $\bigcup_{\xi\in\supp\mu}B_\ep(\xi)$) the material resembles its ``discrete'' nature and the energy behaves like the cardinality of the atoms in the core. Here, (up to a scaling) $\ep$ represents an integer multiple of the lattice spacing

 For screw dislocations in antiplane elasticity it is well known that the displacement  is a multiple of the purely axial Burgers vector times a phase (see, e.g. \cite{VGD}). Let us denote by $w$ a suitable normalization of this vertical displacement and introduce the order parameter $u:=e^{2\pi \imath w}$. Then,  the first addendum in the right-hand side of \eqref{def:cra_intro} is nothing but a suitable normalization of the stored elastic energy outside the cores.
As for the ``plastic term'' $|\mu|(\Omega)$\,, it does not play any role in the asymptotic analysis as $\ep\to 0^+$ and serves only to guarantee compactness; indeed, without such a term a distribution $\mu$ with  $\Omega_\ep(\mu)=\varnothing$ would pay zero energy.

We notice that the scalar field $w$ cannot belong to $H^1(\Omega_\ep(\mu))$\,, since its deformation gradient has a nontrivial circulation condition on a curve enclosing a singularity located at $x^0$; from a modeling point of view, the function $w$ must be in $SBV^2(\Om)$, hence showing a nonzero jump on some  one-dimensional set $S_w$ {which has $x^0$ as one of its endpoints}\,. This jump set represents the (section of) a slip plane, while its amplitude $[w]$ corresponds to the circulation of the absolutely continuous gradient $\nabla w$\,, which is quantized (i.e., it is integer valued), and represents the mismatch between the atomic lattices above and below $S_w$\,.

\medskip
The most celebrated model used to describe topological singularities is, of course, (GL). In such a case, the energy depends only on the order parameter $u\in H^1(\Omega;\R^2)$ which is ``penalized'' (instead of ``constrained'') to take values in $\Ss^1$ and which is defined on the whole $\Omega$\,. The most basic form of the energy is given by  
\begin{equation}\label{def:gl_intro}
\mathcal{E}^{\mathrm{GL}}_{\ep}(u):=\frac{1}{2}\int_{\Omega}|\nabla u|^2\ud x+\frac{1}{\ep^2}\int_{\Omega}\big(1-|u|^2\big)^2\,\ud x\,,
\end{equation}
{where the parameter $\ep$ is called in this case {\it coherence length}.}
Note that the role of the measure $\mu$ in (CR) is actually played in the (GL) model
 by $Ju=\det\nabla u$ which is here a diffuse measure.
 
Let us focus on how the models (CR) and (GL) are able to detect the topological singularities.
Loosely speaking, 
a prototypical topological singularity of degree $z\in\Z\setminus\{0\}$ at a point $x^0\in\Omega$ can be thought of as the
point singularity of a vectorial order parameter $\bar u_\ep:\Omega\to\R^2$ which, in a ball $B_\delta(x^0)$ with $\ep\ll\delta\ll 1$\,, winds around the center as $\big(\frac{x-x^0}{|x-x^0|}\big)^z$\, . The energy of $\bar u_\ep$ diverges at order $|\log\ep|$ as $\ep\to 0^+$\,. As a consequence, to detect the effective
energy cost of finitely many vortex singularities, one needs to study the  (CR) and (GL) energy functionals
at a logarithmic scaling. 
It has been proved (see e.g. \cite{ABO,JS} and references therin)  that a sequence $\{u_\ep\}_\ep$, along which these
energy functionals are equi-bounded, has Jacobians $Ju_\ep$ that, up to a subsequence,
converge in the flat sense (see Section \ref{sec:prelim}) to an atomic measure $\pi\mu=\pi\sum_{i=1}^Iz^{i}\delta_{x^i}$ (here the $x^{i}$ represent the positions of the limiting vortices and $z^i\in\Z\setminus\{0\}$ are their multiplicities). 
The $\Gamma$-limit of $\frac{\mathcal{E}_\ep^{\mathrm{GL}}}{|\log\ep|}$
 with respect to this convergence is then given by $\pi|\mu|$\,. The same analysis has been developed within the (CR) approach in \cite{P}; 
moreover, in \cite{ACP} it is proven that the two models are {\it asymptotically equivalent (in terms of $\Gamma$-convergence)} at energy regime $|\log\ep|^p$ with $p\ge 1$\,.
Variants of the models (CR) and (GL) have been intensively studied; we recall, for instance, discrete linear models for screw dislocations and vortices \cite{AC, P, ADGP, DL} (see also  \cite{ABCDP, ACD} for the case of periodic inhomogeneous materials) and semidiscrete and discrete models for edge dislocations \cite{MSZ, GLP, DGP, ADLPP}. 
\\

In this paper, we are interested in a generalization of the models (CR) and (GL) from a different perspective; specifically, we aim at providing a meaningful energy functional which depends on order parameter $u$ defined on the whole domain $\Omega$, taking values in $\Ss^1$ {and whose approximate gradient is square-integrable in $\Omega$}\,. 
From the point of view of screw dislocations, this means that the jump  $[w]$ of the scalar displacement $w$ cannot\footnote{if it would, in presence of an isolated singularity, the gradient of $u$ would only  be in $L^p$ with $1\le p<2$\,.} be an entire multiple of the Burgers vector (i.e., equivalently, be integer-valued) on the  whole jump set $S_w$\,, since this would lead to a deformation gradient that is non square-integrable. Therefore, the jump $[w]$ of the displacement should exhibit  a transition between integers in a ``small'' subset $\Sigma_{w}$ of $S_w$\,, close to the singularity. 
 In our model this transition has a length of the order $\ep$ and as we let $\ep\to 0^+$ the displacement jump is increasingly ``forced'' to take integer values everywhere on $S_w$\,;
 in other words the total length of $\Sigma_w$  vanishes asymptotically as $\ep\to 0^+$, giving rise to isolated atomic singularities. 
These considerations lead us to consider the following energy
\begin{align}\label{defEne_general_w}
\widehat{\mathcal G}_\ep(w):=\int_\Om\frac12|\nabla w|^2\ud x+\frac1\ep \int_{S_w}W([w])\ud \mathcal H^1\,,
\end{align}   
where the first term is the standard antiplane elastic energy associated to the displacement $w$\,, whereas the second one is a multi-well potential which is null on $\mathbb Z$ and positive elsewhere.  As for instance, one can take the potential 
\begin{equation*}
W(t):=\textrm{dist}^{\alpha}(t,\mathbb Z)\,,\qquad\qquad\textrm{ for some parameter }\alpha\ge 0\,,
\end{equation*}
and hence, passing to $u:=e^{2\pi \imath w}$\,, one arrives at
\begin{equation}\label{potealfa}
W([w]):=|[u]|^\alpha\,.
\end{equation}
{For fixed $\ep$, minimizing \eqref{defEne_general_w} involves a competition between the stored-elastic energy (first term in \eqref{defEne_general_w}) which might be large (if for instance boundary conditions for $u$ are given on $\partial\Omega$) and the creation of a integer displacement mismatch within the otherwise regular crystal lattice (which has an energetical cost given by the  second term in  \eqref{defEne_general_w}).}

Notice that, although potentials as in \eqref{potealfa} are largely used in the context of crack mechanics \cite{BFM}, {with the jump term meaning  the energetic cost of crack creation or propagation, 
} the role of $W$ in our case is a bit different. 
Indeed, here, $W$ penalizes the ``wrong'' mismatch of atomic lattices through the interface $S_w$\,. 
Therefore, passing to $u:=e^{2\pi \imath w}$\,, 
the functional $\widehat{\mathcal G}_\ep$ in \eqref{defEne_general_w} takes the form
\begin{align}\label{defEne_general_w2}
	\mathcal G_\ep(u):=\int_\Om\frac12|\nabla u|^2\ud x+\frac1\ep \int_{S_u}|[u]|^\alpha \ud \mathcal H^1\,.
\end{align}
Clearly, in order to describe the presence of topological singularities, the functional $\mathcal G_\ep$ cannot be restricted to Sobolev functions, but should be defined on the set of $\Ss^1$-valued $SBV^2$ maps\,.

In this paper we focus on the case $\alpha=0$ and define
\begin{equation}\label{defEne_intro}
	\F_\ep(u):= \int_{\Omega}\frac 1 2 |\nabla u|^2\ud x+\frac 1 \ep\Huno(\overline{S}_u)\,.
\end{equation}
{We point out that, at least under some geometric assumptions on $\Om$ and prescribing the set $S_u$ accordingly\footnote{e.g. when $\Om$ is a square, $S_u$ the horizontal axis of $\Om$, and prescribing periodic boundary conditions of $u$.}, one can relate the elastic energy to the $H^{\frac12}$-norm of the jump of $u$, whence leading to 1d models which have been already investigated (see \cite{FG}; related models for different lattice mismatches can be found in \cite{FPP,FPS} and references therein). Related to this setting is \cite{GaMu,CaGa} where a similar analysis is done in 2d.}

 Therefore we aim at studying the asymptotic behavior of $\frac{\F_\ep}{|\log\ep|}$ as $\ep\to 0$\,. 
To this purpose, we first need to introduce a suitable notion of Jacobian for $\Ss^1$-valued $SBV^2$ maps. This task is accomplished in Section \ref{jacobianSBV:sec}, {where we introduce an extension of 
the Jacobian determinant for maps in $BV(\Om;\R^2)\cap L^\infty(\Om;\R^2)$ (see \cite{BrNg1,BrNg2,BOS} and the references therein for fine properties of the Jacobian determinant in Sobolev spaces)}. For a map $u:\Om\rightarrow\R^2$, the distributional Jacobian determinant $\textrm{Det}(\nabla u)$, introduced in \cite{Mor}, is defined as the distribution
\begin{align}\label{Det}
\langle \textrm{Det}(\nabla u),\varphi\rangle_\Om:=\int_\Om\nabla \varphi\cdot \lambda_u\ud x,
\end{align}
where \begin{align}\label{lambda_u}
	\lambda_u:=\frac12\big(-u^1\frac{\partial u^2}{\partial x_2}+u^2\frac{\partial u^1}{\partial x_2},u^1\frac{\partial u^2}{\partial x_1}-u^2\frac{\partial u^1}{\partial x_1}\big);
\end{align}
 it is well-defined under suitable summability assumptions on $u$ ensuring that 
 $\lambda_u\in L^1(\Om;\R^2)$. {Moreover, denoting by $j(u)$ the {\it current} associated to $u$\,, i.e.,  $j(u):=\frac12(u^1\nabla u^2-u^2\nabla u^1)$\,, one has $j(u)^\perp=\lambda_u$ and   it is easy to check that
 $$
 \mathrm{Det}(\nabla u)=-\mathrm{Div}\lambda_u=\mathrm{curl}  j(u)={\frac12\mathrm{curl}(\nabla w)},
 $$
 holds in the sense of distributions, and where $\nabla w$ is the approximate gradient of $w$.}
 In general $\textrm{Det}(\nabla u)$ does not coincide with the pointwise Jacobian determinant $\textrm{det}(\nabla u)$, but they coincide under suitable assumptions on $u$ (see \cite{Mul}). When $u\in BV(\Om;\R^2)$ it is possible to extend definition \eqref{Det} if one shows that $\lambda_u$ is a Radon measure uniquely determined by $u$\,. Such a fact is trivial for smooth maps.
 Hence, arguing by approximation, it is possible to show that if $\{v_k\}_{k\in\N}$ is a sequence of smooth functions approaching $u$ w.r.t. the strict topology of  $BV$\,, then  the distribution $\lambda_u$ is uniquely determined as the limit of the distributions $\lambda_{v_k}$\,. This is a consequence of the results by Jerrard and Jung who showed the existence and uniqueness of a so-called {\it minimal lifting} of $\mad u$, for all $u\in BV(\Om;\R^2)$ (see Theorem \ref{thm:JJ} and \cite{Jerrard} for details).
 
We emphasize that in this planar case with maps with values in $\R^2$, using the aforementioned result of Jerrard and Jung, an extension of distributional determinant for maps of bounded variation has been recently obtained in \cite{M}, under the additional assumption that a map $u$ has finite relaxed graph area (where relaxation is done w.r.t. the strict topology of $BV$). This means that $u$ is approximable strictly in $BV$ by smooth maps $v_k$ with uniformly bounded graphs area (see \cite{BCS,M} for the relaxed area functional). 
 {This implies that the Cartesian current in $\Om\times \R^2$ obtained as the limit of the graphs of the approximating $v_k$ is uniquely determined, which guarantees, in turn, that   the Jacobian of $u$ is always a Radon measure with finite total variation (see \cite{M} for details).} 

However, {requiring only that $u\in BV(\Omega;\R^2)\cap L^\infty(\Omega;\R^2)$\,, the weak Jacobian determinant $Ju$ of $u$ (defined in \eqref{def_dTu_2dim}) is, in general, not a measure,}
but a mere distribution with finite flat norm. 
Moreover, although we will employ the notion of weak Jacobian only for $\Ss^1$- valued maps defined on a planar domain, we point out that the definition of weak Jacobian can be extended to more general (possibly unbounded) functions with bounded variation $u:U\subset\R^n\rightarrow \R^m$, also in dimension and codimension greater than two, allowing to give a notion of weak $2\times2$-minors of $\mad u$\,. {This is stated and proved  in  Theorem \ref{det_teo2}, which also shows that the notion of weak Jacobian extends the distributional determinant in the case of unbounded Sobolev maps.
}

\medskip

Once introduced our weak notion of Jacobian determinant, our main result is Theorem \ref{mainthm}, where we state  compactness and $\Gamma$-convergence for the functionals $\frac{\F_\ep}{|\log\ep|}$\,. As in the classical (CR) and (GL) models, the compactness property shows that, if  $\{u_\ep\}_\ep\subset SBV^2(\Om;\Ss^1)$ with 
\begin{equation}\label{bound_ene_intro}
\sup_{\ep>0}\frac{\mathcal F_\ep(u_\ep)}{|\log\ep|}<+\infty\,,
\end{equation}
then, up to a subsequence, the weak Jacobians $Ju_\ep$ converge in the flat sense to a measure $\pi\mu:=\pi\sum_{n=1}^Nz^n\delta_{x^n}$ with $z^n\in \mathbb Z\setminus\{0\}$\,. 
The starting point of the compactness analysis is that $\Huno(\overline{S}_{u_\ep})\le C\ep|\log\ep|$\,, which guarantees that $\overline{S}_{u_\ep}$ can be covered by a finite (for fixed $\ep$) family of balls such that the sum of its radii vanishes as $\ep\to 0$\,. 
Starting from this family one can use 
the {\it ball construction} technique introduced in \cite{S,J} in order to provide lower bounds for the (CR) and (GL) models. However, such a tool is not enough in order to guarantee compactness in our case.
Indeed, in (GL) the corresponding bound as in \eqref{bound_ene_intro}, implies that $|Ju_\ep|(\Omega)\le C|\log\ep|$\,, which in turns  allows the authors of \cite{AP} to prove that the flat distance between the diffuse Jacobian of $u_\ep$ and the atomic measure provided by the ball construction tends to zero.
In order to bypass the  control on the total variation of the Jacobians (which in our case are not even measures), we adopt the following strategy.
 In order to avoid to estimate the flat distance of the singularities accumulating at the boundary, 
 in our main Theorem \ref{mainthm} we assume that the jump sets of the maps $u_\ep$ are all contained in a fixed set $\Omega'\subset\subset\Omega$\,. In  Theorem \ref{mainthmW} we use this setting  to provide the same result with an applied displacement  boundary condition on $\partial\Omega$ but this time with possibly $S_u\cap\partial\Omega\neq\varnothing$\,.
  Moreover, in order to deal with zero-average clusters contained in $\Omega$\,, we provide an iterative machinery that allows to modify the balls provided by the ball construction in order to get a ``good bound'' of $|\mad u_\ep|$ on the boundary of small regions containing these clusters. {From a physical standpoint these zero-average clusters are related to the so-called {\it statistically-stored} dislocations in single crystals, as opposed to the ({\it geometrically necessary})
  dislocations required to accommodate the incompatible deformations.}
\medskip

Some comments on our approach are in order.
A natural question is whether in \eqref{defEne_intro} one can drop the closure on $S_u$\,, replacing $\Huno(\overline{S}_u)$ with $\Huno(S_u)$\,. This creates some problems in the proof of compactness since, in such a case, the jump set, whose $\Huno$ measure vanishes as $\ep\to 0^+$\,, can be dense in $\Omega$\,.
 In terms of our proof, this means that $S_{u_\ep}$ can be covered by a family of {\it countably many}, not necessarily {\it finitely many}, 
 balls whose sum of the radii vanishes as $\ep\to 0^+$\,.
 Since the ball construction \cite{S,J} requires that the number of initial balls is finite, the problem above requires an extension of the ball construction procedure to the case of countably many initial balls. 
 
Another issue concerns the case   $\alpha>0$ in \eqref{defEne_general_w2}. Also such a problem cannot be studied in full analogy with the one considered here, in view of the 
lack of a-priori control of the $\Huno$-measure of the jump set in that case.
Furthermore, we do believe that our approach could be extended to higher energy regimes as well as
{to} the vectorial case of edge dislocations. This last issue is much more delicate, since it requires the extension of our approach to maps with vectorial lifting. 
{Finally, an extension of our approach to the 3d setting \cite{CGO} deserves future investigations; in such a case, techniques coming from the  theory of integral and Cartesian currents could be exploited, as successfully done in \cite{Hoc, CGM, SVG3, SVG1, SVG2, GMS}.}

\medskip
\textbf{Plan of the paper.} The paper is organized as follows: In Section \ref{sec:prelim} we introduce the notation and preliminary notion used in the sequel. In Section \ref{jacobianSBV:sec} we extend the notion of distributional Jacobian determinant to some classes of maps of bounded variation (see Theorems \ref{det_teo} and \ref{det_teo2}). In Section \ref{sec:model} we finally pass to the $\Gamma$-convergence analysis of the functional \eqref{defEne_intro} and show our main result in Theorem \ref{mainthm}.
\medskip

\textsc{Acknowledgements:}
We thank Marcello Ponsiglione for stimulating discussions on the subject of the paper.
LDL and RS are members of the Gruppo Nazionale per l'Analisi Matematica, la Probabilit\`a e le loro Applicazioni (GNAMPA) of the Istituto Nazionale di Alta Matematica (INdAM).

\section{{Preliminary results}}\label{sec:prelim}
{In this section we collect some preliminary notion on the flat norm of measures and currents, as well as some properties of $BV$ functions that will be used throughout the paper.}
\medskip
\paragraph{\bf Multi-indices.} Let $m,n\geq2$ be two fixed integers.
Let $I\subset\{1,\dots,n\}$ and $J\subset\{1,\dots,m\}$ be multi-indices (in particular, they have a specific order). We denote by $\widehat I$ the ordered set $\{1,\dots,n\}\setminus I$ and $\widehat J=\{1,\dots,m\}\setminus J$. We also denote by $\sigma(I,\widehat I)$ and $\sigma(J,\widehat J)$ the signs of the permutations $(I,\widehat I)\in S(n)$ and $(J,\widehat J)\in S(m)$, where $(I,\widehat I)$ and $(J,\widehat J)$ are seen as the sets $\{1,\dots,n\}$ and $\{1,\dots,m\}$ with a precise order. We will usually deal with the case in which $I=\{i,i'\}$ and $J=\{j,j'\}$ are multi-indices consisting of exactly to distinct elements.
\medskip
\paragraph{\bf {Flat norm of Radon measures}} 
{Let $n\ge 1$ be an integer and let $U\subset\R^n$ be an open and bounded set.} We denote by $\mathcal M_b(U)$ the space of Radon measures on $U$\,. If $\mu\in \mathcal M_b(U)$, we denote by $|\mu|(U)$ the total variation of $\mu$\,. 
We also introduce the concept of flat norm  of a measure $\mu$, noted by $\|\mu\|_{\flt}$:
\begin{align}\label{flat_norm_mu}
	{\|\mu\|_{\flt}:=\sup_{{\newatop{\ffi\in \Cc^{0,1}(U)}{\|\ffi\|_{C^{0,1}(U)}\le 1}}}\int_U\varphi \ud\mu\,,}
\end{align}
{Here and below, the Lipschitz norm $\|\ffi\|_{C^{0,1}(U)}$ is defined by
$$
\|\ffi\|_{C^{0,1}(U)}:=\|\ffi\|_{L^\infty(U)}+\sup_{\newatop{x,y\in U}{x\neq y}}\frac{|\varphi(x)-\varphi(y)|}{|x-y|}\,.
$$}
By a density argument we easily see that the supremum in  \eqref{flat_norm_mu} can be equivalently computed among smooth and compactly supported (in $U$) functions $\varphi$ with {$\|\ffi\|_{C^{0,1}(U)}\le 1$}\,.
\medskip
\paragraph{{\bf Flat norm of $k$-currents.}} {Let $n\ge 2$ be an integer and let $U\subset\R^n$ be an open set}.  For every  $k\in\N$ with $0\leq k\leq n$\,, we denote by $\mathcal D^k(U)$ the topological vector space of smooth and compactly supported $k$-forms on $U$, and by $\mathcal D_k(U)$ its dual, i.e., the space of $k$-currents on $U$. The canonical basis of 1-vectors in {$\R^n$ is denoted} by $\{e_1,\dots,e_n\}$, and its dual basis of 1-covectors by $\{\mathrm{d} x^1,\dots,\mathrm{d} x^n\}$\,.

The mass $|T|$ of a current $T\in \mathcal D_k(U)$ is defined as
$$
|T|=\sup\{\langle T,\omega\rangle:\; \omega\in \mathcal D^k(U),\|\omega\|_{L^\infty}\leq1\}\,.
$$
As for measures, we define the {\it flat norm} of a current $T\in\mathcal D_k(U)$ in $U$ by
\begin{equation}\label{flatcurr}
	\|T\|_{\flt,U}:=\sup_{\newatop{\omega\in\mathcal{D}^k(U)}{\|\omega\|_{F,U}\le 1}}\langle T,\omega\rangle,
\end{equation}
where
$$
\|\omega\|_{F,U}:=\|\omega\|_{L^\infty(U)}+\|\mathrm{d}\omega\|_{L^\infty(U)}\,.
$$
In the special case that $T$ is a $0$-current and has finite mass, then it can be standardly identify with a measure, and  the flat norm of $T$ coincides with the flat norm of the measure $T$ {defined in \eqref{flat_norm_mu}}.
\medskip
\paragraph{\bf Jacobian and $\Ss^1$-valued maps} 
Given a map $u\in W^{1,1}(U;\R^2)\cap L^\infty (U;\R^2)$ we recall that  $Ju=\textrm{Det}(\nabla u)$ is defined by \eqref{Det}. 
Notice that $\textrm{Det}(\nabla u)$ is well-defined also if $u\in W^{1,\frac43}(U;\R^2)$, since by Sobolev embedding theorem we have $\lambda_u\in L^1(U;\R^2)$ (see \eqref{lambda_u}).

In the sequel we will use the fact that a function $u\in H^1(U;\Ss^1)$ satisfies $\textrm{Det}(\nabla u)=0$ in the sense of distributions. 

Moreover, if $u\in H^1(U\setminus\overline B;\Ss^1)$, where $B\subset U$ is a ball, then, integrating by parts,
$$\int_{U\setminus \overline B}\lambda_u\cdot \nabla \varphi\ud x=\int_{\partial B}\lambda_u\cdot \nu \varphi \ud\mathcal H^1=\int_{\partial B}j(u)\cdot \tau \varphi \ud\mathcal H^1,\qquad \forall \varphi\in \Cc^\infty(\Om),$$
where $\nu$ is the inner normal vector to $\partial B$, $\tau=-\nu^\perp$ is the counter-clockwise tangent normal vector to $\partial B$, and $j(u)=\frac12(u^1\nabla u^2-u^2\nabla u^1)=-\lambda_u^\perp$. Notice that $j(u)\cdot \tau=\frac12(u^1\frac{\partial u^2}{\partial \tau}-u^2\frac{\partial u^1}{\partial \tau})$ on $\partial B$.

We recall that $\deg(u,\partial B)\in \mathbb Z$ is defined as
\begin{align}
\deg(u,\partial B):=\frac1\pi\int_{\partial B}j(u)\cdot \tau \ud\mathcal H^1=\frac1\pi\int_{\partial B}\lambda_u\cdot \nu \ud\mathcal H^1\,,\label{def_grado}
\end{align}
whenever $u\in H^{\frac12}(\partial B;\Ss^1)$.

\medskip

\paragraph{\bf Vector-valued $BV$ functions} 
Here we recall some basic facts on vector-valued $BV$-functions that will be used throughout the paper.
We refer to the monograph \cite{AFP} for a complete analysis of $BV$-functions.

Let $m,n\ge 2$ be two integers and let
$\Omega\subset\R^n$ be an open set. Let moreover $u\in BV(\Om;\R^m)$. Then, there is an $\mathcal H^{n-1}$-rectifiable set $S_u\subset\Om$ such that {$\mathcal H^{n-1}$}-a.e. $x\in \Om\setminus S_u$ is a Lebesgue point for $u$\,. For every such $x\in \Om\setminus S_u$\,, we denote by $\overline u(x)\,$, the Lebesgue value of $u$ at $x$\,, so that $\overline u$ turns out to be well-defined $\mathcal H^{n-1}$-almost everywhere on $\Om\setminus S_u$\,. 
The set $S_u$ has the following property: For $\mathcal H^{n-1}$-a.e. $x\in S_u$ there exists a unit vector $\nu(x)\in \R^n$, which is orthogonal to the approximate tangent space to $S_u$ at $x$, and there exist $u^+(x),u^-(x)\in \R^m$ such that the maps 
$\nu,u^+,u^-$ are $\mathcal H^{n-1}$-measurable on $S_u$, and 
$$u^+(x)=\textrm{aplim}_{y\rightarrow x,\;(y-x)\cdot\nu(x)>0}u(y),\qquad u^-(x)=\textrm{aplim}_{y\rightarrow x,\;(y-x)\cdot\nu(x)<0}u(y)\,,
$$  
for $\mathcal H^{n-1}$-a.e. $x\in S_u$.

The distributional gradient $\mad u$ of $u$ reads as
$\mad u=\mad^D u+\mad^S u$, where $\mad^S u$ is the jump part of the gradient provided by
$$
\mad^S u=[u]\otimes \nu\ud\mathcal H^{n-1}\,,
$$
with $[u]:=u^+-u^-$ ($u^\pm$ being the traces of $u$ on the two sides of $S_u$), and $\mad^D u=\mad u\res(\Om\setminus S_u)$ is the diffuse part of the gradient, which satisfies $\mad^D u(A)=0$ if $\mathcal H^{n-1}(A)<\infty$. In turn,
$\mad^D u=\mad^a u+\mad^c u$ splits into two mutually singular measures,  $\mad^a u$ which is  absolutely continuous with respect to {the Lebesgue measure} $\mathcal L^n$, and the Cantor part $\mad^c u$\,.

Eventually, we extend the definition of $\overline u$, which has been defined so far only $\mathcal H^{n-1}$-a.e.  on $\Om\setminus S_u$, by setting
\begin{equation}\label{conunmezzo}
\overline u(x)=\frac12(u^+(x)+u^-(x)), \qquad x\in S_u\,.
\end{equation}
In this way, $\overline u$ is well-defined $\mathcal H^{n-1}$-a.e. on $S_u$ and, in general, $\overline u$ is {well-defined} $|\mad u|$-a.e. in $\Om$.
{Furthermore, for every $\theta\in [0,1]$ and $\mathcal H^{n-1}$-a.e. $x\in S_u$ we define the function $\overline{u}^\theta$ as 
	\begin{equation}\label{contheta}
	\overline{u}^\theta(x):=\theta u^+(x)+(1-\theta)u^-(x)\,,
	\end{equation}
	so that $\overline{u}\equiv\overline{u}^{\frac 1 2}$\,.}	
Finally, we say that a sequence $\{u_k\}_{k\in\N}\subset BV(\Omega;\R^m)$ {\it strictly converges} to $u$ in  $BV(\Om;\R^m)$ (as $k\to +\infty$), and we write $u_k\strictly u$  (as $k\to +\infty$), if $u_k\to u$ (strongly) in $L^1(\Omega;\R^m)$ and $|\mad u_k|(\Omega)\to |\mad u|(\Omega)$  (as $k\to +\infty$)\,.
\medskip
\paragraph{\bf Minimal lifting} Here we deal with maps $v:\Om\rightarrow\R^m$, where $\Om\subset\R^n$ is a bounded open set. We denote by $x=(x_1;\dots;x_n)\in \R^n$ the coordinates in $\R^n$, and $y=(y^1;\dots;y^m)$ those in $\R^m$.
Given $v=(v^1;\ldots;v^m)\in C^1(\Om;\R^m)\cap W^{1,1}(\Om;\R^m)$\,, for every $i=1,\dots,m$ and $j=1,\dots,n$\,,  we define $(\mu_v)^i_j\in \M_b(\Om\times\R^m)$ as the only measure satisfying
\begin{align}\label{muij}
	\int_{\Om\times\R^m}\phi(x,y)\ud(\mu_v)^i_j=\int_\Om\phi(x,v(x))\frac{\partial v^i}{\partial x_j}\ud x\,,\qquad\textrm{for every }\phi\in\Cc(\Omega\times\R^m)\,.
\end{align} 
In \cite{Jerrard} the following result is proved. 
\begin{theorem}\label{thm:JJ}
	Let $u\in BV(\Om;\R^m)$\,. Then, for every $i=1,\dots,m$ and $j=1,\dots,n$\,, there exists a measure $(\mu_u)^i_j\in \mathcal M_b(\Om\times\R^m)$ such that the following holds: if $\{v_k\}_{k\in\N}\subset{C^1}(\Om;\R^m)\cap W^{1,1}(\Om;\R^m)$ satisfies $v_k\strictly v$ (as $k\to +\infty$)\,, then $(\mu_{v_k})^i_j\weakstar(\mu_u)^i_j$ and the $\R^{m\times n}$-valued measure $\mu_{v_k}:=\big((\mu_{v_k})_j^i\big)_{\newatop {i=1,\ldots,m}{j=1,\ldots,n}}$ converges tightly to $\mu_u:=\big((\mu_u)_j^i\big)_{\newatop {i=1,\ldots,m}{j=1,\ldots,n}}$ (as $k\to +\infty$). Furthermore,
	for every $i=1,\dots,m$ and $j=1,\dots,n$\,, it holds
	\begin{align}\label{explicituDu}
		\int_{\Om\times\R^m}\phi(x,y)\ud(\mu_u)^i_j=&\int_{\Om\setminus S_u}\phi(x,\overline u(x))\ud(D^D_ju^i)(x)\nonumber\\
		&+\int_{S_u}\int_{0}^{1}\phi(x,\overline u^\theta(x))\ud\theta\ud(D^S_ju^i)(x)\,,	
	\end{align}
	for every $\phi\in \Cc(\Omega\times\R^m)$\,.
\end{theorem}
As a consequence of Theorem \ref{thm:JJ}, it turns out that the $\R^{m\times n}$-valued measure $\mu_u$ is uniquely determined by $u\in BV(\Om;\R^m)$ and is called {\it minimal lifting} of $\mad u$. Moreover, it easily follows that 
if $\{u_k\}_{k\in\N}\subset BV(\Om;\R^m)$ satisfies $u_k\strictly u$ in $BV(\Om;\R^m)$ (as $k\to +\infty$)\,, then $\mu_{u_k}\rightarrow\mu_u$ tightly. 
\section{{Weak $2\times2$-minors of gradients and jacobian determinants of $BV$ functions}} \label{jacobianSBV:sec}
We now see how the notion of minimal lifting {provided by Theorem \ref{thm:JJ}} allows to define,  in a weak sense, the $2\times2$-minors of gradient of suitable $BV$-maps.

{Let $m\ge 2$ and $n\ge 2$ be two integers and let $\Omega\subset\R^n$ be an open and bounded set.
} Given $v\in C^1(\Omega;\R^m)\cap W^{1,1}(\Om;\R^m)$, testing \eqref{muij} with $\phi(x,y)=\ffi(x)\psi(y)$\,, where $\ffi\in\Cc(\Omega)$ and $\psi\in\Cc(\R^m)$\,, we have
\begin{align*}
	\int_{\Om\times\R^m}\ffi(x)\psi(y)\ud(\mu_v)^i_j=\int_\Om\ffi(x)\psi(v(x))\frac{\partial v^i}{\partial x_j}\ud x\,.
\end{align*}
It follows that, if also $v\in  L^\infty(\Om;\R^m)$\,, for every $h=1,\ldots,m$\,, we can take 
$\psi:=\psi^h$ with 
$\psi^h\in \Cc(\R^m)$ such that
$\psi^h(y)=y^h$ in ${\overline{B}_{\|v\|_{L^\infty}}(0)\subset\R^m}$\,, thus obtaining
\begin{align}\label{vDv}
	\int_{\Om\times\R^m}\ffi(x)\psi^h(y)\ud(\mu_v)^i_j=\int_\Om\ffi(x)v^h(x)\frac{\partial v^i}{\partial x_j}\ud x\,.
\end{align}
Now, let $u\in BV(\Om;\R^m)\cap L^\infty(\Om;\R^m)$ and $\{v_k\}_{k\in\N}\subset C^1(\Omega;\R^m){\cap W^{1,1}(\Om;\R^m)\cap L^\infty(\Omega;\R^m)}$ be such that $v_k\strictly u$ as $k\to +\infty$, and $\|v_k\|_{L^\infty}\leq C$ for all $k\geq1$\,.
Then, it is easy to see that  for every $i,h=1,\ldots,m$ and for every $j=1,\ldots,n$, it holds
\begin{equation}\label{intermezzo}
v_k^h(x)\frac{\partial v_k^i}{\partial x_j}\ud x\weakstar (\nu_u)^{i,h}_j\qquad\textrm{as }k\to +\infty\,,
\end{equation}
where the measure $(\nu_u)^{i,h}_j$ is defined by the formula
\begin{align}\label{def_nu}
	\int_\Om\ffi(x)\ud(\nu_u)^{i,h}_j=\int_{\Om\times\R^m}\ffi(x)y^h\ud(\mu_u)^i_j\qquad \forall \ffi\in \Cc(\Om)\,,
\end{align}
for all $i,h\in\{1,\dots,m\}$, $j\in \{1,\dots,n\}$\,. 
{More precisely, we have the following result.}
\begin{corollary}\label{coro}
	Let $u\in BV(\Om;\R^m)\cap L^\infty(\Om;\R^m)$, then there exists a unique measure $\nu_u \in \mathcal M_b(\Om;\R^{m\times n\times m})$ such that, whenever {$\{v_k\}_{k\in\N}\subset C^1(\Om;\R^m)\cap W^{1,1}(\Om;\R^m)\cap L^\infty(\Om;\R^m)$}
	satisfies $\|v_k\|_{L^\infty}\leq C<+\infty$ for all $k\geq1$ and 
	$v_k\strictly u$ in $BV(\Om;\R^m)$\,, then {$v_k\otimes\nabla v_k\rightarrow  \nu_u$}\,, {where $(\nu_u)_{j}^{i,h}$ is defined by \eqref{def_nu} for every $i,h=1,\ldots,m$ and $j=1,\ldots,n$\,.}  
	Finally, if $\{u_k\}_{k\in\N}\subset BV(\Om;\R^m)\cap L^\infty(\Om;\R^m)$ with 
	\begin{equation}\label{Linftybound}
	\|u_k\|_{L^\infty(\Omega;\R^m)}\leq C
	\end{equation}
	 for some constant $C>0$, and $u_k\strictly u$ in $BV(\Om;\R^m)$\,, then 
	\begin{equation}\label{convdebstar}
	\nu_{u_k}\weakstar\nu_u\qquad\textrm{in }\mathcal M_b(\Om;\R^{m\times n\times m})\,.
	\end{equation}
\end{corollary}
\begin{proof}
	The first part of the statement follows straightforwardly from {the discussion above (see formulas \eqref{vDv}, \eqref{intermezzo}, \eqref{def_nu})}. Let us comment on the last claim, and assume $u_k\strictly u$ with $\{u_k\}_{k\in\N}\subset BV(\Om;\R^m)\cap L^\infty(\Om;\R^m)$. 
	{Let $\ffi\in\Cc(\Omega)$\,, $i,h\in\{1,\ldots,m\}$ and $j\in\{1,\ldots,n\}$\,.
	Let moreover $\psi^h\in\Cc(\R^m)$ be such that $\psi^h=y^h$ inside $B_{C}(0)\subset\R^m$\,, where $C$ is the constant in \eqref{Linftybound}.
By Theorem \ref{thm:JJ}, we get 
	\begin{equation}\label{20220427}
	\lim_{k\to +\infty}\int_{\Om\times\R^m}\ffi(x)\psi^h(y)\ud(\mu_{u_k})^i_j= \int_{\Om\times\R^m}\ffi(x)\psi^h(y)\ud(\mu_{u})^i_j\,.
	\end{equation}
	On the other hand, in view of \eqref{Linftybound}, by \eqref{def_nu}, we have that
	\begin{equation*}
	\begin{aligned}
	\int_{\Om\times\R^m}\ffi(x)\psi^h(y)\ud(\mu_{u_k})^i_j=&\int_{\Omega}\ffi(x)\ud(\nu_{u_k})^{i,h}_j\\
	\int_{\Om\times\R^m}\ffi(x)\psi^h(y)\ud(\mu_{u})^i_j=&\int_{\Omega}\ffi(x)\ud(\nu_u)^{i,h}_j,
	\end{aligned}
	\end{equation*}
	which, together with \eqref{20220427}, implies that
	\begin{equation*}
	\int_{\Omega}\ffi(x)\ud(\nu_{u_k})^{i,h}_j\overset{k\to +\infty}{\longrightarrow}\int_{\Omega}\ffi(x)\ud(\nu_{u})^{i,h}_j\,.
	\end{equation*}
	By the arbitrariness of $\ffi$ we get \eqref{convdebstar}.}
\end{proof}
In view of Corollary \ref{coro}, {for every $u\in BV(\Omega;\R^m)\cap L^\infty(\Omega;\R^m)$\,,} we are allowed to adopt the following notation
 \begin{align*}
	[u^h\mad_{j}u^i]:=(\nu_u)^{i,h}_j,\;\;\;\;i,h\in\{1,\dots,m\},\;\;\;\;j\in \{1,\dots,n\}\,,
\end{align*}
and
$$
[u\otimes\mad u]:=\nu_u:=\Big((\nu_u)^{i,h}_j\Big)_{\newatop{i,h\in\{1,\ldots,m\}}{j\in\{1,\ldots,n\}}}\,.
$$
{Let $u\in BV(\Omega;\R^m)\cap L^\infty(\Omega;\R^m)$
and} let us fix four indices $i,i'\in \{1,\dots,m\}$ and $j,j'\in \{1,\dots,n\}$\,. {We set} $I:=\{i,i'\}$, $J:=\{j,j'\}$, and consider the measures
 $\lambda^I_j,\lambda^I_{j'}\in \mathcal M_b(\Om)$ given by 
\begin{equation}\label{def_lambda}
\begin{aligned}
	\lambda^I_j:=&\frac12\big([u^i\mad_ju^{i'}]-[u^{i'}\mad_ju^{i}]\big)=\frac12((\nu_u)_j^{i',i}-(\nu_u)_j^{i,i'})\,,\\
	\lambda^I_{j'}:=&\frac12\big([u^i\mad_{j'}u^{i'}]-[u^{i'}\mad_{j'}u^{i}]\big)=\frac12((\nu_u)_{j'}^{i',i}-(\nu_u)_{j'}^{i,i'})\,.
\end{aligned}
\end{equation}
%
%
{We define the current $(T_u)_{I,J}\in\mathcal D_1(\Om)$ as}
\begin{align}\label{def_Tu}
	(T_u)_{I,J}:=&\sigma^I_J(\lambda_j^Ie_{j'}-\lambda_{j'}^Ie_j)\nonumber\\
	=&\sigma_J^I\frac12\big([u^i\mad_ju^{i'}]-[u^{i'}\mad_ju^{i}]\big)e_{j'}-\sigma_J^I\frac12\big([u^i\mad_{j'}u^{i'}]-[u^{i'}\mad_{j'}u^{i}]\big)e_j,
\end{align}
where $\sigma_J^I:=\sigma(I,\widehat I)\sigma(J,\widehat J)\in \{-1,1\}$\,. The $1$-current $(T_u)_{I,J}$ acts on $1$-forms $\omega\in \mathcal D^1(\Om)$ with $\omega=\sum_h \omega_h\ud x^h$, $\omega_h\in \Cc^\infty(\Om\times\R^m)$, as
\begin{align}\label{def_Tu2}
	(T_u)_{I,J}(\omega)&=\sigma_J^I\Big(\int_\Om\omega_{j'}\ud\lambda_j^I-\int_\Om \omega_j\ud\lambda_{j'}^I\Big)\nonumber\\
	&=\frac{\sigma_J^I}{2}\int_\Om \omega_{j'}\ud([u^i\mad_ju^{i'}]-[u^{i'}\mad_ju^{i}])-\frac{\sigma_J^I}{2}\int_\Om\omega_{j}\ud([u^i\mad_{j'}u^{i'}]-[u_{i'}\mad_{j'}u_{i}]).
\end{align}
A consequence of Corollary \ref{coro} is that $(T_u)_{I,J}$
is well-defined (for all $u\in BV(\Om;\R^m)\cap L^\infty(\Om;\R^m)$), it is a Radon measure on the space of compactly supported and  continuous $1$-forms and is weakly continuous with respect to strict {convergence} of $u_k$ to $u$ in $BV(\Om;\R^m)$, provided that $u_k$ are uniformly bounded in $L^\infty(\Om;\R^m)$.
Therefore we arrive at the following {result.}
\begin{theorem}\label{det_teo}
	For all $u\in BV(\Om;\R^m)\cap L^\infty(\Om;\R^m)$ and for all $I= \{i,i'\}\subset\{1,\dots,n\}$, $J= \{j,j'\}\subset\{1,\dots,m\}$, the current $\partial (T_u)_{I,J}\in \mathcal D_0(\Om)$ is uniquely determined as the boundary of the current $(T_u)_{I,J}$ given  by \eqref{def_Tu2}. Moreover,
	 \begin{equation}\label{flat_est}
	 \|\partial(T_u)_{I,J}\|_{\flt,\Omega}\le C\|u\|_{L^\infty(\Omega;\R^m)}|\mad u|(\Om)\,,
	 \end{equation}
	 for some universal constant $C>0$ (independent of $u$)\,.
\end{theorem} 
\begin{proof}
The first part of the statement is clear. Let us prove \eqref{flat_est}:
Let $\varphi\in \Cc^\infty(\Om)$ and consider the $0$-form $\omega=\varphi$\,, whose differential is $\mathrm d\omega=\sum_h\frac{\partial \varphi}{\partial x_h}\ud x^h$. Hence, by definition of boundary of a current 
\begin{align}\label{formulaabove}
	\partial(T_u)_{I,J}(\omega)&=(T_u)_{I,J}(\mathrm d\omega)=\int_\Om \frac{\partial\varphi}{\partial x_{j'}}\ud\lambda_j^I-\int_\Om\frac{\partial\varphi}{\partial x_{j}}\ud\lambda_{j'}^I\nonumber\\
	&=\frac{\sigma_J^I}{2}\int_\Om \frac{\partial\varphi}{\partial x_{j'}}\ud([u^i\mad_ju^{i'}]-[u^{i'}\mad_ju^{i}])-\frac{\sigma_J^I}{2}\int_\Om\frac{\partial\varphi}{\partial x_{j}}\ud([u^i\mad_{j'}u^{i'}]-[u^{i'}\mad_{j'}u^{i}])\,,
\end{align}
from which \eqref{flat_est} follows by definition of flat norm.
\end{proof}
We can then introduce the definition of weak $2\times2$-minors of $Du$:
\begin{definition}\label{def_2x2minor}
Let $u\in BV(\Om;\R^m)\cap L^\infty(\Om;\R^m)$ and let  $I= \{i,i'\}\subset\{1,\dots,n\}$, $J= \{j,j'\}\subset\{1,\dots,m\}$. Then we define $M^J_I(Du)\in \mathcal D'(\Om)$ the $2\times2$-minor of $Du$ with rows in $I$ and columns in $J$ as the distribution
\begin{align}\label{def_minor}
	\langle M^J_I(Du),\varphi\rangle_\Om:&= \int_\Om \frac{\partial\varphi}{\partial x_{j'}}\ud\lambda_j^I-\int_\Om\frac{\partial\varphi}{\partial x_{j}}\ud\lambda_{j'}^I\nonumber\\
	&=\frac{\sigma_J^I}{2}\int_\Om \frac{\partial\varphi}{\partial x_{j'}}\ud([u^i\mad_ju^{i'}]-[u^{i'}\mad_ju^{i}])-\frac{\sigma_J^I}{2}\int_\Om\frac{\partial\varphi}{\partial x_{j}}\ud([u^i\mad_{j'}u^{i'}]-[u^{i'}\mad_{j'}u^{i}])\,,
\end{align}
for all $\varphi\in \mathcal D(\Om)$,
where $\lambda^I_j$ and $\lambda^I_{j'}$ are defined in \eqref{def_lambda}.
\end{definition} 
%
%
%
%
{Notice that, if $v\in C^1(\Omega;\R^m)\cap W^{1,1}(\Om;\R^m)\cap L^\infty(\Omega;\R^m)$\,, by \eqref{formulaabove}, \eqref{def_minor}, and \eqref{vDv}, we get} 
\begin{align*}
	M^J_I(\nabla v)=&\frac{\sigma_J^I}{2}\int_\Om \frac{\partial\varphi}{\partial x_{j'}}\Big(v^i\frac{\partial v^{i'}}{\partial x_j}-v_{i'}\frac{\partial v^{i}}{\partial x_j}\Big)\ud x-\frac{\sigma_J^I}{2}\int_\Om\frac{\partial\varphi}{\partial x_{j}}\Big(v^i\frac{\partial v^{i'}}{\partial x_{j'}}-v^{i'}\frac{\partial v^{i}}{\partial x_{j'}}\Big)\ud x\nonumber\\
	&=\sigma_J^I\int_\Om\varphi(x)\Big(\frac{\partial v^{i}}{\partial x_j}(x)\frac{\partial v^{i'}}{\partial x_{j'}}(x)-\frac{\partial v^{i}}{\partial x_{j'}}(x)\frac{\partial v^{i'}}{\partial x_j}(x) \Big)\ud x\,,
\end{align*} 
namely, the distribution $M^J_I(\nabla v)$ coincides with the $(2\times2)$-subdeterminant of the matrix $\nabla v$ with rows in $I$ and columns in $J$\,.
\medskip
\subsection{An extension of the $2\times2$-minors for more general $BV$ maps}
In this section we refine the result in Theorem \ref{det_teo} {weakening} the hypothesis that $u\in L^\infty(\Om;\R^m)$\,.
To this purpose, we preliminarily notice that,  if $u\in BV(\Om;\R^m)\cap L^\infty(\Om;\R^m)$\,, we can take $\phi(x,y)=\ffi(x)y^{h}$ in \eqref{explicituDu}\,, with $\ffi\in\Cc(\Omega)$ and $h\in\{1,\ldots,m\}$\,.
From this expression, using \eqref{def_nu}, we obtain
\begin{align}\label{explicitnu}
	\int_{\Om}\ffi(x)\ud(\nu_u)^{i,h}_j&=\int_{\Om\times\R^m}\ffi(x)y^h\ud(\mu_u)^i_j\nonumber\\
	&=\int_{\Om\setminus S_u}\ffi(x)\overline u^h(x)\ud(\mad^D_ju^i)+\int_{S_u}\ffi(x)\int_{0}^{1} (\overline{u}^\theta)^h(x)\ud\theta \ud(\mad^S_ju^i)\nonumber\\
	&=\int_{\Om\setminus S_u}\ffi(x)\overline u^h(x)\ud(\mad_ju^i)+\int_{S_u}\ffi(x) \overline{u}^h(x)\ud(\mad_ju^i)\,,	
\end{align}
where $\overline{u}^\theta=((\overline{u}^\theta)^1;\ldots; (\overline{u}^\theta)^m)$ is defined in \eqref{contheta}.
Let now $u\in BV(\Om;\R^m)$ and assume that
\begin{equation}\tag{P}\label{property}
\begin{aligned}
&\textrm{ For some  $i,i'\in \{1,\dots,m\}$ with $i\neq i'$ and  $j,j'\in\{1,\dots,n\}$ with $j\neq j'$\,, it holds}\\
&		\phantom{\textrm{ For every $i,i'\in \{1,\dots,m\}$}}y^i\in L^1(\Om\times\R^m,(\mu_u)^{i'}_j)\cap L^1(\Om\times\R^m,(\mu_u)^{i'}_{j'})\,,\\
&		\phantom{\textrm{ For every $i,i'\in \{1,\dots,m\}$}}y^{i'}\in L^1(\Om\times\R^m,(\mu_u)^{i}_j)\cap L^1(\Om\times\R^m,(\mu_u)^{i}_{j'})\,,\\
&\textrm{{where $L^1(\Om\times\R^m,\mu)$ denotes the space of $L^1$ functions with respect to the measure $\mu$\,.}}
\end{aligned}
\end{equation}
We emphasize that if $u\in BV(\Om;\R^m)\cap L^\infty(\Om;\R^m)$ then property \eqref{property} is readily satisfied, since $\supp{\mu_u}$ is bounded in $\Om\times\R^m$. 
{
In Theorem \ref{det_teo2} below, we will show that property \eqref{property} is enough to ensure the well-posedness of the definition of the current \eqref{def_Tu} as well as to guarantee that it and its boundary have flat norm. 
To this purpose, we preliminarily notice that, in view of \eqref{property}, by \eqref{explicituDu},} for every $\phi\in \Cc(\Om\times\R^m)$ of the form $\phi(x,y)=\ffi(x)\psi(y)|y^h|$\,, it holds
\begin{equation}\label{consequence_P}
\begin{aligned}
	&\int_{\Om\times\R^m}\ffi(x)\psi(y)|y^h|\ud(\mu_u)^i_j\\
	=&\int_{\Om\setminus S_u}\ffi(x)\psi(\overline u(x))|\overline u^h(x)|\ud(\mad_ju^i)+\int_{S_u}\ffi(x)\int_{0}^{1}\psi(\overline{u}^\theta(x))| (\overline{u}^h)^\theta(x)|\ud\theta\ud(\mad_ju^i)\\
	\leq&\int_{\Om\setminus S_u}|\overline u^h(x)|\ud|\mad_ju^i|+\int_{S_u}\int_{0}^{1}| (\overline{u}^\theta)^h(x)|\ud\theta\ud|\mad_ju^i|.
\end{aligned}
\end{equation}
Therefore, we can take the supremum of the left-hand side among all $\ffi\in \Cc(\Omega;[-1,1])$ and $\psi\in \Cc(\R^m;[-1,1])$\,, and by standard density arguments we get  
\begin{align}\label{totalvariation}
	\int_{\Om\times\R^m}|y^h|\ud|(\mu_u)^i_j|\leq \int_{\Om\setminus S_u}|\overline u^h(x)|\ud|\mad_ju^i|+\int_{S_u}\int_{0}^{1}| (\overline{u}^\theta)^h(x)|\ud\theta\ud|\mad_ju^i|.
\end{align}
In particular, if $u\in BV(\Om;\R^m)$ satisfies \eqref{property}, then the measure $y^h\cdot (\mu_u)_j^i$ has finite total variation which is bounded by the right-hand side of \eqref{totalvariation}.
Now we show that \eqref{totalvariation} holds actually with equality.

\begin{lemma}
	Let $u\in BV(\Om;\R^m)$ and let $i,i'\in \{1,\dots,m\}$ with $i\neq i'$ and  $j,j'\in\{1,\dots,n\}$ with $j\neq j'$ be such that \eqref{property} holds. Then for $h\in \{i,i'\}$ it holds
	\begin{align}\label{totalvariation_eq}
		\int_{\Om\times\R^m}|y^h|\ud|(\mu_u)^i_j|= \int_{\Om\setminus S_u}|\overline u^h(x)|\ud|\mad_ju^i|+\int_{S_u}\int_{0}^{1}| (\overline{u}^\theta)^h(x)|\ud\theta\ud|\mad_ju^i|.
	\end{align}
In particular it follows
\begin{align}\label{integrability}
	\overline u^{h}\in L^1(\Om\setminus S_u,|\mad^D_ju^i|), \qquad \overline u^h\in L^1(S_u,|\mad^S_ju^i|)\,,
\end{align}
where  the $\overline u$ is defined on $S_u$ by \eqref{conunmezzo}.
\end{lemma}
\begin{proof}
Fix $h,i,j$ and $\ep>0$, and  choose a compact set $S_u^\ep\subset S_u$ such that 
\begin{equation}\label{piccolo1}
|\mad_ju^i|(S_u\setminus S_u^\ep)=|\mad_j^Su^i|(S_u\setminus S_u^\ep)< \frac \ep 2\,.
\end{equation}
Thanks to the compactness of $S^\ep_u$, for all $\delta>0$ small enough, we can choose an open neighborhood $U_\delta^\ep$ of $S_u^\ep$ such that 
\begin{equation}\label{piccolo2}
|\mad^D_ju^i|(U_\delta^\ep)<\frac \ep 2\,,
\end{equation}
and $U_\delta^\ep\searrow S^\ep_u$ as $\delta\rightarrow0$\,.
{Notice that, by \eqref{piccolo1}, for all $\delta$ we have
\begin{equation}\label{piccolo3}
|\mad_j^Su^i|(S_u\setminus U_\delta^\ep)\le |\mad_j^Su^i|(S_u\setminus S^\ep_u)<\frac\ep 2\,,
\end{equation} 
whereas, by \eqref{piccolo2}, we get
\begin{equation}\label{piccolo4}
|\mad_ju^i|(U^\ep_\delta\setminus S_u)=|\mad_j^Du^i|(U^\ep_\delta\setminus S_u)<\frac \ep 2\,.
\end{equation}
}
Let  $\ffi_{{U^\ep_\delta}}\in \Cc(U_\delta^\ep;[-1,1])$ and $\ffi_{{U^\ep_\delta}}^c\in \Cc(\Om\setminus  {{\overline U^\ep_\delta}};[-1,1])$ and set $\ffi:=\ffi_{{U^\ep_\delta}}+\ffi_{{U^\ep_\delta}}^c$\,. Moreover, let $\psi\in\Cc(\R^m;[-1,1])$\,. 
By the equality in \eqref{consequence_P}, using that $\sup\{|y|:y\in \supp\psi\}=:\widehat C(\psi)<+\infty$, we obtain
 \begin{align}\label{laprima}
 	&\int_{\Om\times\R^m}\ffi(x)\psi(y)|y^h|\ud(\mu_u)^i_j\\ \nonumber
 	=&\int_{U_\delta^\ep\setminus S_u}\ffi_{{U^\ep_\delta}}(x)\psi(\overline u(x))|\overline u^h(x)|\ud(\mad_ju^i)+\int_{\Om\setminus (U_\delta^\ep\cup S_u)}\ffi_{{U^\ep_\delta}}^c(x)\psi(\overline u(x))|\overline u^h(x)|\ud(\mad_ju^i)\\ \nonumber
 	&+\int_{ {{U^\ep_\delta}}\cap S_u}\ffi_{{U^\ep_\delta}}(x)\int_{0}^{1}\psi(\overline{u}^\theta(x))| (\overline{u}^\theta)^h(x)|\ud\theta\ud(\mad_ju^i)\\ \nonumber
 	&+\int_{ (\Om\setminus {{U^\ep_\delta}})\cap S_u}\ffi_{{U^\ep_\delta}}^c(x)\int_{0}^{1}\psi(\overline{u}^\theta(x))| (\overline{u}^\theta)^h(x)|\ud\theta \ud(\mad_ju^i)\\ \nonumber
 	\geq& -|\mad_ju^i|({{U^\ep_\delta}}\setminus S_u)\|\,{\widehat C}(\psi)+\int_{\Om\setminus ({{U^\ep_\delta}}\cup S_u)}\ffi_{{U^\ep_\delta}}^c(x)\psi(\overline u(x))|\overline u^h(x)|\ud(\mad_ju^i)\\ \nonumber
 	&-|\mad^S_ju^i|(S_u\setminus S_u^\ep)\|\,{\widehat C}(\psi)+\int_{ {{U^\ep_\delta}}\cap S_u^\ep}\ffi_{{U^\ep_\delta}}(x)\int_{0}^{1}\psi(\overline{u}^\theta(x))| (\overline{u}^\theta)^h(x)|\ud\theta\ud(\mad _ju^i)\\ \nonumber
 	&-|\mad^S_ju^i|(S_u\cap (\Om\setminus {{U^\ep_\delta}}))\|\,{\widehat C}(\psi)\,.
 \end{align}
Therefore,  taking in the right-hand side of \eqref{laprima} the supremum over $\ffi_{{U^\ep_\delta}}\in  \Cc(U_\delta^\ep;[-1,1])$\,, $\ffi_{{U^\ep_\delta}}^c\in \Cc(\Om\setminus  {{\overline U^\ep_\delta}};[-1,1])$\,, and assuming $\psi\geq0$, by \eqref{piccolo1}-\eqref{piccolo4}, we deduce
\begin{equation}\label{20220428_1}
\begin{aligned}
	\int_{\Om\times\R^m}|y^h|\ud{|(\mu_u)^i_j|}
	\geq &-2\ep\,{\widehat C}(\psi)+\int_{\Om\setminus (U_\delta^\ep\cup S_u)}|\psi(\overline u(x))\overline u^h(x)|\ud|\mad_ju^i| \\
	&+\int_{ U_\delta^\ep\cap S_u^\ep}\int_{0}^{1}|\psi(\overline{u}^\theta(x)) (\overline{u}^\theta)^h(x)|\ud\theta\ud|\mad_ju^i|\,,
\end{aligned}
\end{equation}
for all $\psi\in\Cc(\R^m;[0,1])$\,.
Letting $\delta\rightarrow0$ in \eqref{20220428_1}, for every $\psi\in\Cc(\R^m;[0,1])$ we get
\begin{align*}
	\int_{\Om\times\R^m}|y^h|\ud{|(\mu_u)^i_j|}\geq& -2\ep\,{\widehat C}(\psi)+\int_{\Om\setminus S_u}|\psi(\overline u(x))\overline u^h(x)|\ud|\mad_ju^i|   \nonumber\\
	&\;\;\;+\int_{ S_u^\ep}\int_{0}^{1}|\psi(\overline{u}^\theta(x)) (\overline{u}^\theta)^h(x)|\ud\theta \ud|\mad_ju^i|
\end{align*}
and, by the arbitrariness of $\ep>0$\,, we conclude
\begin{equation}\label{20220428_2}
\begin{aligned}
	\int_{\Om\times\R^m}|y^h|\ud{|(\mu_u)^i_j|}\geq& \int_{\Om\setminus S_u}|\psi(\overline u(x))\overline u^h(x)|\ud|\mad_ju^i|\\
	&+\int_{ S_u}\int_{0}^{1}|\psi(\overline{u}^\theta(x)) (\overline{u}^\theta)^h(x)|\ud\theta\ud|\mad_ju^i|\,,
\end{aligned}
\end{equation}
for every $\psi\in\Cc(\R^m;[0,1])$\,.
Let us choose a non-increasing compactly supported continuous function $\Psi:[0,\infty)\rightarrow[0,1]$ such that $\Psi(0)=1$, and, for every $M>0$ we set $\psi_M(\cdot)=\Psi(\frac{\cdot}{M})$\,. 
Taking $\psi(y):=\psi_M(|y^h|)$ in \eqref{20220428_2}, and letting $M\to+\infty$, by Fatou Lemma we get
\begin{align*}
	\int_{\Om\times\R^m}|y^h|\ud{|(\mu_u)^i_j|}\geq& \int_{\Om\setminus S_u}|\overline u^h(x)|\ud|\mad_ju^i|+\int_{ S_u}\int_{0}^{1}| (\overline{u}^\theta)^h(x)|\ud\theta \ud|\mad_ju^i|\,,
\end{align*}
 which, together with \eqref{totalvariation}, proves \eqref{totalvariation_eq}.
\end{proof}

We will make use of property \eqref{integrability} in order to prove the following theorem.
%
\begin{theorem}\label{det_teo2}
Let $u\in BV(\Om;\R^m)$, let $I=\{i,i'\}\subset\{1,\dots,m\}$, and $J=\{j,j'\}\subset \{1,\dots,n\}$ {with $i\neq i'$ and $j\neq j'$}\,. Assume that $u$ satisfies hypothesis \eqref{property} for $i,i'$, and $j,j'$. Then for all $i,h\in I$, $j\in J$ there exists a unique measure $(\nu_u)^{i,h}_j$ enjoying \eqref{explicitnu} and having finite total variation. As a consequence, the current $(T_u)_{I,J}\in \mathcal D_1(\Om)$ given by \eqref{def_Tu} is well-defined,  has finite mass, and 
\begin{align*}
	\|\partial (T_u)_{I,J}\|_{\flt,\Om}\leq |(\lambda_u)^{I}_j|(\Om)+|(\lambda_u)^{I}_{j'}|(\Om)\,,
\end{align*}
{where the measures $(\lambda_u)^{I}_j$ and $(\lambda_u)^{I}_{j'}$ are defined in \eqref{def_lambda}.}
Moreover, if $\{u_k\}_{k\in\N}\subset BV(\Om;\R^m)$ is a sequence of maps satisfying \eqref{property} and converging strictly to $u$ in $BV(\Omega;\R^m)$
with
\begin{equation}\label{oranumerata}
\sup_{k\in\N}\Big\{\sum_{h=j,j'}\big(|y^{i}(\mu_{u_k})^{i'}_h|(\Om\times\R^m)+|y^{i'}(\mu_{u_k})^{i}_h|(\Om\times\R^m)\big)\Big\}<C\,,
\end{equation}
for some $C>0$\,,
then, as $k\to +\infty$\,, $(\nu_{u_k})^{i,h}_j\weakstar(\nu_u)^{i,h}_j$ for all $i,h\in I$, $j\in J$, and $(T_{u_k})_{I,J}$ converges to $(T_u)_{I,J}$ in $\mathcal D_1(\Om)$.	 
\end{theorem}
\begin{proof}
{Fix $i\in I$, $h=i'$, and $j\in J$. We start by showing that the measures $(\nu_u)^{i,h}_j$ and $(\nu_u)^{i,h}_{j'}$ provided by \eqref{explicitnu} are well-defined and have finite total variation in $\Om\times\R^m$\,.}	
	
	\medskip
	
	\textbf{Step 1:}
	For every $N>0$ let $\eta_N:\R\rightarrow[-N,N]$ be the function defined by $\eta_N(t):=(-N)\vee t\wedge N$\,.
	 We introduce the standard truncation of $u$ at level $N$ as {$u_N:=(\eta_N(u^{h}))_{h=1,\ldots,m}$\,.} 
	 For all $N>0$\,, since $u^{h}_N\in L^\infty(\Om)$ and in view of \eqref{explicitnu}, the measure $(\nu_j^{i,i'})_N$ defined by
	 $$
	 (\nu_j^{i,i'})_N:=(\nu_{u_N})_j^{i,i'}=[u_N^{i'}\mad_ju_N^i]\,,
	 $$
	is well-defined as a Radon measure in $\mathcal M_b(\Om)$\,.
	By \eqref{def_nu} and \eqref{explicitnu}, {for all $\ffi\in\Cc(\Omega)$\,,} we can write 
	\begin{equation}\label{nu_troncata}
	\begin{aligned}
		\int_\Om\ffi(x)\ud(\nu_j^{i,i'})_N=&\int_{\Om\setminus S_u}\ffi(x)\overline u_N^{i'}(x)\ud (\mad_ju_N^i)+\int_{S_u}\varphi(x) \overline u_N^{i'}(x) \ud(\mad_ju_N^i)
	\end{aligned}
	\end{equation}
	By \cite[Theorem 7, pag. 486]{GMS1} $\eta_N$ is differentiable at $u^i(x)$ for $\mad_ju^i$-a.e. $x\in \Om\setminus S_{u}$, and  
	\begin{equation*}
		\mad_ju_N^i=\left\{\begin{array}{ll}
		\displaystyle \mad_j\eta_N(u^i)\mad_ju^i & \text{ in }\Om\setminus S_{u}\,,\\[1mm]
		\displaystyle \big((u^i_N)^+-(u^i_N)^-\big)\nu^j\cdot\mathcal H^{n-1}&\text{ in }S_{u}\,,
		\end{array}
		\right.
	\end{equation*}
where
	\begin{equation*}
	\mad_j\eta_N(u^i)=\left\{
	\begin{array}{ll}
	\displaystyle 1&\textrm{ in }\Om_N:=\{x\in \Om\setminus S_{u}:|u^i(x)|<N\}\\[1mm]
	 \displaystyle 0&\textrm{ in }R_N:=\{x\in \Om\setminus S_{u}:|u^i(x)|\geq N\}\,,
	\end{array}
	\right.
	\end{equation*}
	with 
	\begin{equation}\label{convinsi}
	\Om_N\nearrow(\Om\setminus S_{u})\qquad\textrm{ as $N\rightarrow\infty$ up to a }|\mad_ju^i|\textrm{-negligible set.}
	\end{equation}
	 Hence, we can write 
	\begin{align*}
		\mad_ju_N^i=\begin{cases}
			\mad_ju^i &\text{ on }\Om_N\\
			0&\text{ on }R_N.
		\end{cases}
	\end{align*}
	\medskip
	
	\textbf{Step 2:}
	We claim that, as $N\to+\infty$\,, $(\nu_j^{i,i'})_N\weakstar\nu_j^{i,i'}$ where $\nu_j^{i,i'}$ is uniquely determined by the formula
	\begin{equation}\label{claim_step2}
	\begin{aligned}
		\int_\Om\ffi(x)\ud\nu_j^{i,i'}=&\int_{\Om\times\R^m}\ffi(x)y^{i'}\ud(\mu_u)^i_j\\
		=&\int_{\Om\setminus S_u}\ffi(x)\overline u^{i'}(x)\ud (\mad_ju^i)+\int_{S_u}\ffi(x) \overline u^{i'}(x) \ud(\mad_ju^i)\,,
	\end{aligned}
	\end{equation}
	for all $\ffi\in \Cc(\Om)$\,. 
	To prove this, we start by observing that for every $N>0$\,, \eqref{nu_troncata} implies 
	\begin{equation}\label{122}
	\begin{aligned}
		|(\nu_j^{i,i'})_N|(\Om)\leq& \int_{\Om\setminus S_{u}}|\overline u_N^{i'}|\ud |\mad_ju_N^i|+\int_{S_{u}}\frac12| (u_N^{i'})^++(u_N^{i'})^-| \ud |\mad_ju_N^i|\\
		\leq&  \int_{\Om\setminus S_{u}}|\overline u^{i'}|\ud|\mad_ju^i|+\int_{S_{u}}\frac12| (u^{i'})^++(u^{i'})^-| \ud|\mad_ju^i|<\infty\,,
	\end{aligned}
	\end{equation}
	where the last inequality is a consequence of property \eqref{property} (see \eqref{integrability}) and the last but one inequality follows from 
	\begin{equation}\label{siusaanchedopo}
	|\overline u_N^h|\leq |\overline u^h|\,,\qquad |(u_N^{i'})^++(u_N^{i'})^-|\leq |(u^{i'})^++(u^{i'})^-|\qquad\qquad\textrm{for all }N>0\,, h\in\{i,i'\}\,.
	\end{equation}
	 By \eqref{122}, up to subsequences, 
	 \begin{equation}\label{convemisure}
	 (\nu_j^{i,{i'}})_N\weakstar\nu_j^{i,{i'}}\,, \qquad\qquad\textrm{as measures as $N\to+\infty$\,,}
	 \end{equation}
	  for some $\nu_j^{i,i'}\in \mathcal M_b(\Om)$\,.
	  Let $\ffi\in\Cc(\Omega)$ be fixed; by \eqref{nu_troncata}, for every $N>0$ it holds
	 \begin{equation}\label{nu_troncata2}
	\begin{aligned}
		\int_\Om\ffi(x)\ud(\nu_j^{i,i'})_N=&\int_{\Om_N}\ffi(x)\overline u_N^{i'}(x)\ud (\mad_ju^i)\\
		&+\int_{S_{u}}\ffi(x)\frac12((u_N^{i'})^++(u_N^{i'})^-)((u_N^i)^+-(u_N^i)^-) \nu^j\ud \mathcal H^{n-1}\,;
	\end{aligned}
	\end{equation}
	therefore, using once again \eqref{integrability},\eqref{siusaanchedopo}, \eqref{convinsi} and the fact that
	$$
	|(u_N^i)^+-(u_N^i)^-|\leq |(u^i)^+-(u^i)^-|\qquad\qquad\textrm{$\mathcal H^1$-a.e. on $S_u$, for every }N>0\,,
	$$
	 by the Dominated Convergence Theorem, we conclude that the right-hand side of \eqref{nu_troncata2} tends to (as $N\to +\infty$)
	\begin{align*}
		\int_{\Om\setminus S_{u}}\ffi(x)\overline u^{i'}(x)\ud(\mad_ju^i)+\int_{S_{u}}\ffi(x)\frac12( (u^{i'})^++(u^{i'})^-) 	\ud(\mad_ju^i)\,.
	\end{align*}
	This fact, together with \eqref{convemisure} and \eqref{nu_troncata2}, implies that
	\begin{equation*}
	\int_\Om\ffi(x)\ud\nu_j^{i,i'}
		=\int_{\Om\setminus S_u}\psi(x)\overline u^{i'}(x)\ud (\mad_ju^i)+\int_{S_u}\ffi(x) \overline u^{i'}(x) \ud(\mad_ju^i)\,,
	\end{equation*}
for all $\ffi\in\Cc(\Omega)$\,.	
	To conclude the claim it remains to show that also the first equality in \eqref{claim_step2} is satisfied.
	 By \eqref{def_nu}, for every $N>0$\,, we have
	\begin{align}\label{def_nuN}
		\int_\Om\ffi(x)\ud(\nu^{i,i'}_j)_N=\int_{\Om\times\R^m}\ffi(x)y^{i'}\ud(\mu_{u_N})^i_j\qquad \forall \ffi\in \Cc(\Om)\,,
	\end{align}
	and, by Theorem \ref{thm:JJ}\,, $(\mu_{u_N})^i_j\weakstar(\mu_{u})^i_j$ in $\mathcal M_b(\Om\times\R^m)$\,. 
{It follows that $y^{i'}(\mu_{u_N})^i_j\rightharpoonup y^{i'}(\mu_{u})^i_j$ as distributions in $\mathcal D(\Om\times \R^m)$\,.}
	Moreover, by \eqref{122} and \eqref{def_nuN}, we deduce that
	the sequence $\{|y^{i'}(\mu_{u_N})^i_j|(\Om)\}_N$ is uniformly bounded, and hence, up to subsequences, $y^{i'}(\mu_{u_N})^i_j\weakstar\tau$ for some $\tau\in\mathcal M_b(\Om\times\R^m)$\,. Therefore, $\tau=y^{i'}(\mu_{u})^i_j$\,, the whole sequence $\{y^{i'}(\mu_{u_N})^i_j\}_N$ converges to $y^{i'}(\mu_{u})^i_j$ (by the Urysohn property)\,, and	
	\begin{align*}
		\int_\Om\ffi(x)\ud\nu^{i,i'}_j=\int_{\Om\times\R^m}\ffi(x)y^{i'}\ud(\mu_{u})^i_j\qquad \forall \ffi\in \Cc(\Om)\,,
	\end{align*}
	thus concluding the proof of  \eqref{claim_step2} .
	%
	\medskip
	
	\textbf{Step 3:} It remains to show the last claim in the statement.
	{By \eqref{oranumerata}}, up to  a subsequence, $y^{i'}(\mu_{u_{k}})^i_j\weakstar\tau$ for some $\tau\in\mathcal M_b(\Om\times\R^m)$\,. Arguing as in Step 2, we get that $\tau=y^{i'}(\mu_{u})^i_j$ and that the whole sequence $y^{i'}(\mu_{u_{k}})^i_j$ converges to $\tau$\,.
	As a consequence, we deduce that $(\nu_{u_k})^{i,i'}_j\weakstar(\nu_{u})^{i,i'}_j$, explicitly given by 
	\begin{align*}
		\int_{\Om\times\R^m}\ffi(x)y^{i'}\ud(\mu_{u})^i_j&=\int_{\Om\setminus S_{u}}\ffi(x)\overline u^{i'}(x)\ud(\mad_ju^i)+\int_{S_{u}}\psi(x)\overline u^{i'}(x)\ud(\mad_ju^i)=\int_\Om \ffi(x)\ud(\nu_u)_j^{i,i'},	
	\end{align*}
	for all $\ffi\in \Cc(\Om)$\,. 
	By the very definitions of $\lambda_j^I$ and $\lambda_{j'}^I$ in \eqref{def_lambda} and of $T_u$ in \eqref{def_Tu}, the claim easily follows.	
\end{proof}
We can then give the following definition, which extends Definition \ref{def_2x2minor}.

\begin{definition}
Let  $I= \{i,i'\}\subset\{1,\dots,n\}$, $J= \{j,j'\}\subset\{1,\dots,m\}$ with $i\neq i'$ and $j\neq j'$, and let $u\in BV(\Om;\R^m)$ be such that \eqref{property} holds for $I$ and $J$. Then we define $M^J_I(Du)\in \mathcal D'(\Om)$ the $2\times2$-minor of $Du$ with rows in $I$ and columns in $J$ as the distribution in \eqref{def_minor}.
\end{definition}
\begin{remark}\label{remDD}
\rm{
	As a byproduct of the proof of Theorem \ref{det_teo2}, we deduce that, if $u\in BV(\Om;\R^m)$ satisfies hypothesis \eqref{property}, then the measures $y^i(\mu_u)^{i'}_j$ and $(\nu_u)_j^{i',i}$ take the form
	\begin{align*}
		\int_{\Om\times\R^m}\phi(x,y)y^i\ud(\mu_{u})^{i'}_j=&\int_{\Om\setminus S_{u}}\phi(x,\overline u(x))\overline u^i(x)\ud(\mad_ju^{i'})\nonumber\\
		&+\int_{S_{u}}\int_{0}^{1}\phi(x,\overline{u}^\theta)(\overline{u}^\theta)^i(x) \ud\theta\ud(\mad_ju^{i'}),
	\end{align*}
	for all $\phi\in \Cc(\Om\times\R^m)$, and 
	\begin{align*}
		\int_{\Om\times\R^m}\varphi(x)y^{i}\ud(\mu_{u})^{i'}_j&=\int_{\Om\setminus S_{u}}\varphi(x)\overline u^{i}(x)\ud(\mad_ju^{i'})+\int_{S_{u}}\varphi(x)\overline u^{i}(x)\ud(\mad_ju^{i'}),
	\end{align*}
	for all $\varphi\in \Cc(\Om)$.
	If $u\in L^p(\Om;\R^m)\cap W^{1,q}(\Om;\R^m)$, with $\frac1p+\frac1q=1$ (including the case $p=\infty$, $q=1$), property \eqref{property} is readily satisfied, and then
	the previous formulas read
	\begin{align*}
		\int_{\Om\times\R^m}\phi(x,y)y^i\ud(\mu_{u})^{i'}_j=&\int_{\Om}\phi(x,u(x)) u^i(x)\frac{\partial u^{i'}}{\partial x_j}(x)\ud x,
	\end{align*}
	for all $\phi\in \Cc(\Om\times\R^m)$, and
	\begin{align*}
		\int_\Om \varphi(x)\ud(\nu_u)_j^{i',i}	=\int_{\Om\times\R^m}\varphi(x)y^{i}\ud(\mu_{u})^{i'}_j=\int_{\Om}\varphi(x)u^{i}(x)\frac{\partial u^{i'}}{\partial x_j}(x)\ud x,
	\end{align*}
	for all $\varphi\in \Cc(\Om)$\,.
	In particular, the measure $(\nu_u)_j^{i',i}$ belongs to $L^1(\Om)$ and coincides almost everywhere with the function $u^{i}\frac{\partial u^{i'}}{\partial x_j}$\,.
	}
\end{remark}

\subsection{The $2\times2$ dimensional case}\label{2drmk}

Let us now restrict ourselves to the two dimensional case $n=2$, namely $\Om\subset\R^2$ and let also $m=2$. In this case the only possible choice for $I$ and $J$ is $I=J=\{1,2\}$, and the current $(T_u)_{I,J}$ will be simply denoted by $T_u$; furthermore, we denote the distribution $M^J_I(Du)$ by $Ju$. We also denote  $(\lambda_u)^{1,2}_1=:\lambda_1$, and $(\lambda_u)^{1,2}_2=:\lambda_2$. Explicitly,
\begin{align}\label{def_Tu_2dim}
	(T_u)(\omega)&=\int_\Om\omega_{2}\ud\lambda_1-\int_\Om \omega_1\ud\lambda_{2}\nonumber\\
	&=\frac{1}{2}\int_\Om \omega_{2}\ud([u^1\mad_1u^{2}]-[u^{2}\mad_1u^{1}])-\frac{1}{2}\int_\Om\omega_{1}\ud([u^1\mad_{2}u^{2}]-[u^{2}\mad_{2}u^{1}])\,,
\end{align}
for all $\omega=\omega_1\ud x^1+\omega_2\ud x^2\in \mathcal D^1(\Om)$, and for all $\varphi\in \mathcal D^0(\Om)$,
\begin{align}\label{def_dTu_2dim}
	\partial(T_u)(\varphi)&=\langle Ju,\varphi\rangle_\Om\nonumber\\
	&=\frac{1}{2}\int_\Om \frac{\partial\varphi}{\partial x_{2}}\ud([u^1\mad_1u^{2}]-[u^{2}\mad_1u^{1}])-\frac{1}{2}\int_\Om\frac{\partial\varphi}{\partial x_{1}}\ud([u^1\mad_{2}u^{2}]-[u^{2}\mad_{2}u^{1}])\,.
\end{align}
Assume now that $u\in L^p(\Om;\R^m)\cap W^{1,q}(\Om;\R^m)$, with $\frac1p+\frac1q=1$; in view of Remark \ref{remDD}, the last expression reads
\begin{align}
	\langle Ju,\varphi\rangle_\Om=\frac{1}{2}\int_\Om \frac{\partial\varphi}{\partial x_{2}}\Big(u^1\frac{\partial u^2}{\partial x_1}-u^2\frac{\partial u^1}{\partial x_1}\Big)\ud x-\frac{1}{2}\int_\Om\frac{\partial\varphi}{\partial x_{1}}\Big(u^1\frac{\partial u^2}{\partial x_2}-u^2\frac{\partial u^1}{\partial x_2}\Big)\ud x,
\end{align}
which coincides with the definition of the distributional determinant of $\nabla u$. In particular, the distribution $Ju$ extends the definition of distributional Jacobian.

\begin{remark}\label{mucci_rmk}
{\rm In \cite{M} the existence and well-posedness of $Ju=\partial (T_u)$ is obtained under the additional condition that the function $u$ can be approximated (strictly in $BV$) by a sequence of maps $v_k\in C^1(\Om;\R^2)$ satisfying the following condition: There exists a constant $C>0$, independent of $k$, such that 
$$\int_\Om|Jv_k|\ud x\leq C\qquad \forall k\in \mathbb N.$$  
This is equivalent to require that the Jacobian total variation functional relaxed w.r.t. the strict topology of $BV$ is finite. In turn, this is equivalent to require that the relaxed area functional of $v_k$ is finite, and then that the mass of the Cartesian currents $\mathcal G_{v_k}$ with underlying maps $v_k$ have equibounded masses.

Under this assumption, it turns out that also $Ju=\partial(T_u)$ is a Radon measure with finite total variation. In particular it turns out that, whenever $u_k\in  BV(\Om;\R^2)\cap L^\infty(\Om;\R^2)$ are such that 
\begin{align}\label{cond_flat}
&u_k\rightarrow u\qquad \text{strictly in }BV(\Om;\mathbb R^2),\nonumber\\
&|Ju_k|(\Om)\leq C\qquad \forall k\in \mathbb N,
\end{align}
for some constant $C>0$ independent of $k$, then
\begin{align}\label{conv_meas}
Ju_k\rightarrow Ju\qquad\text{ weakly star as measures.}
\end{align}}
\end{remark}
We will make use of \eqref{conv_meas} in the proof of the upper bound of Theorem \ref{mainthm} using the following observation.
\begin{remark}\label{flat_rmk}
{\rm Assume that $u_k\in BV(\Om;\R^2)\cap L^\infty(\Om;\R^2)$ satisfy the two conditions in \eqref{cond_flat} and in addition $\textrm{supp}(Ju_k)\subset K\subset\Om$ for some compact set $K$, then \eqref{conv_meas} ensures that 
\begin{align}
Ju_k\rightarrow Ju\qquad\text{ in the flat topology.}
\end{align}
}

\end{remark}

\section{A new approach to topological singularities}\label{sec:model}
Let $\Omega$ be an open bounded subset of $\R^2$ with Lipschitz continuous boundary and let $\Omega'\subset\subset\Omega$ be an open set. We introduce
\begin{equation}\label{admissiblebdrysenzaw}
	\mathcal {AD}(\Om,\Omega'):=\{u\in SBV^2(\Om;\Ss^1)\,:\,\overline{S}_u\subset\overline{\Omega'}\},
\end{equation}
where 
$S_u$  denotes the jump set of $u$\,. 
For every $\ep>0$\,, let $\F_\ep:L^2(\Omega;\R^2)\to [0,\infty]$ be the functional defined by
\begin{equation}\label{defEne}
	\F_\ep(u):=\left\{\begin{array}{ll}
		\displaystyle \int_{\Omega}\frac 1 2 |\nabla u|^2\ud x+\frac 1 \ep\Huno(\overline{S}_u)&\textrm{if }u\in \mathcal{AD}(\Omega,\Omega')\\
		+\infty&\textrm{elsewhere in }L^2(\Omega;\R^2)\,.
	\end{array}
	\right.
\end{equation}
{In what follows, we will adopt also localized versions of the functional $\F_\ep$\,; more precisely, for any $u\in \mathcal{AD}(\Omega,\Omega')$ and for any open set $\Omega'\subset\subset A\subset\subset\Omega$\,, we will denote by $\F_\ep(u;A)$ the functional in \eqref{defEne} with $\Omega$ replaced by $A$\,.
}

Denoting by $\M_b(\Omega)$ the class of Radon measures with finite total variation in $\Omega$\,, we set
$$
X(\Omega):=\Big\{\mu=\sum_{i=1}^Iz^i\delta_{x^i}\in\M_b(\Omega)\,:\,I\in\N\,,\,z^i\in\Z\setminus\{0\}\,,\,x^i\in\Omega\Big\}\,.
$$
For every $u\in \mathcal{AD}(\Om,\Omega')$ we consider the current $T_u\in \mathcal D_1(\Om)$ introduced in \eqref{def_Tu} and in \eqref{def_Tu_2dim} for the $2$-dimensional case, and we denote by $Ju:=\partial T_u$ its boundary (whose expression is given in \eqref{def_dTu_2dim}), namely the Jacobian determinant of $u$\,. We recall that in general $Ju$ is not a measure, but a mere distribution with finite flat norm (see Theorem \ref{det_teo}).

Our main result is the following.
\begin{theorem}\label{mainthm}
	The following $\Gamma$-convergence result holds true.
	\begin{itemize}
		\item[(i)] (Compactness) Let $\{u_\ep\}_{\ep}\subset L^2(\Omega;\R^2)$ be such that 
		\begin{equation}\label{enbound}
			\sup_{\ep>0}\frac{\F_\ep(u_\ep)}{|\log\ep|}\le C,
		\end{equation}
		for some $C>0$\,. Then there exists $\mu\in X(\Omega)$ such that, up to a subsequence, $\|Ju_\ep-\pi\mu\|_{\flt,\Omega}\to 0$ (as $\ep\to 0$).
		\item[(ii)] ($\Gamma$-liminf inequality) For every $\mu\in X(\Omega)$ and for every $\{u_\ep\}_{\ep}\subset L^2(\Omega;\R^2)$ such that $\|Ju_\ep-\pi\mu\|_{\flt,\Omega}\to 0$ (as $\ep\to 0$)\,, it holds
		\begin{equation}\label{liminf}
			\pi|\mu|(\Omega)\le \liminf_{\ep\to 0}\frac{\F_\ep(u_\ep)}{|\log\ep|}\,.
		\end{equation}
		\item[(iii)] ($\Gamma$-limsup inequality) For every $\mu\in X(\Omega)$\,, there exists  $\{u_\ep\}_{\ep}\subset L^2(\Omega;\R^2)$ with $\|Ju_\ep-\pi\mu\|_{\flt,\Omega}\to 0$ (as $\ep\to 0$)\,, such that
		\begin{equation}\label{eq:limsup}
			\pi|\mu|(\Omega)\geq\limsup_{\ep\to 0}\frac{\F_\ep(u_\ep)}{|\log\ep|}\,.
		\end{equation}
	\end{itemize}
\end{theorem}
In order to prove Theorem \ref{mainthm}, we will make use of Theorem \ref{alipons} below which is proven in  \cite[Theorem 3.2]{AP}) (see also  \cite[Theorem 2.4]{DLP}).
To this purpose, we introduce some notation. 

Let $V\subset\R^2$ be an open bounded set with Lipschitz continuous boundary.
For every finite family of pairwise (essentially\footnote{That is, whose closures are mutually disjoint.}) disjoint open balls $\B:=\{B^n\}_{n=1,\ldots,N}$ (with $N\in\N$)
we set 
$V(\B):=V\setminus \bigcup_{n=1}^N\overline{B}^n$ and we denote by $\rad(\B)$ the sum of the radii of the balls $B^n$, namely
$$\rad(\B):=\sum_{n=1}^Nr(B^n)\,,$$
where $r(B)$ denotes the radius of the ball $B$\,.
Moreover, for every $\mu\in X(V)$ of the form
\begin{align}\label{mu_B}
	\mu:=\sum_{n=1}^Nz^n\delta_{x(B^n)}\qquad \text{ with }z^n\in\Z\setminus\{0\}\,,
\end{align}
 we set
\begin{equation*}
	\Ad(\B,\mu,V):=\{u\in H^1(V(\B);\Ss^1)\,:\, \deg(u,\partial B^n)=z^n\textrm{ for every }n=1,\ldots,N\}\,.
\end{equation*}
Here and below, $x(B)$ denotes the center of the ball $B$\,.
For every $\B$ and $\mu$ as above, we set
\begin{equation}\label{def:coreradius}
	F(\B,\mu,V):=\min_{u\in \Ad(\B,\mu,V)}\int_{V(\B)}|\nabla u|^2\ud x\,.
\end{equation}

\begin{definition}[Merging procedure]\label{merging_def}
	\rm{Given a finite family $\B=\{B_{r^i}(x^i)\}_{i=1,\ldots,I}$ ($I\in\N$) of balls in $\R^2$, we define a new family $\widehat{\B}$ as follows. If the closures of two balls in $\B$ are not disjoint, then we replace the two balls with a unique ball which contains both of them and with radius less or equal to the sum of the radii of the original balls. After this, we repeat this replacement recursively, until as all the balls in the family are mutually essentially disjoint. The final family is $\widehat{\B}$.
		The procedure of passing from $\B$ to $\widehat{\B}$ is called {merging procedure} applied to $\B$. Notice that a merging procedure does not increase the sum of all the radii of the balls in the family.
	}
\end{definition}

The following result is proven in \cite[Proposition 2.2]{DLP}.
\begin{proposition}\label{ballconstr}
	{Let $\B$ be a finite family of pairwise essentially disjoint balls in $\R^2$, and let $\mu\in X(V)$ be of the form \eqref{mu_B}.}
	Then, there exists a one-parameter family of open balls $\B(t)$ with $t\ge 0$  such that, setting $U(t):=\bigcup_{B\in\B(t)} B$, the following properties hold true:
	\begin{enumerate}
		\item $\B(0)=\B \, $;
		\item $ U(t_1)\subset U(t_2)$  for any $0\le t_1<t_2 \, $;
		\item the balls in $\B(t)$ are pairwise (essentially) disjoint;
		\item for any $0\le t_1<t_2$ and for any open set $U \subseteq \R^2 \,$, 
		\begin{equation*}
			F(\B,\mu,U\cap(U(t_2)\setminus U(t_1)))\ge\pi\sum_{\newatop{B\in\B(t_2)}{B\subseteq U}}|\mu(B)|\log\frac{1+t_2}{1+t_1} \, ;
		\end{equation*} 
		\item  $\displaystyle \sum_{B\in \B(t)}r(B)\le(1+t)\sum_{B\in \B}r(B)$, where $r(B)$ denotes the radius of the ball $B \, .$
	\end{enumerate}
\end{proposition}
For every $\B$ and $\mu$ as above, we set $\Ccal(1):=\{B\in\B(1)\,:\,\overline{B}\subset V\}$ and we define 
\begin{equation}\label{tilde}
	\widetilde\mu:=\sum_{B\in\Ccal(1)}\mu(B)\delta_{x(B)}\,.
\end{equation}
We can now state the crucial result on which the proof of Theorem \ref{mainthm} is based. 
\begin{theorem}\label{alipons}
	Let $V$ be a bounded  open set with Lipschitz boundary. For every $\ep>0$ let $\B_\ep:=\{B_\ep^n\}_{n=1,\ldots,N_\ep}$ (with $N_\ep\in\N$) be a (finite) family of pairwise (essentially) disjoint open balls with $\rad(\B_\ep)\to 0$ as $\ep\to 0$ and let $\mu_\ep:=\sum_{n=1}^{N_\ep}z_{\ep}^n\delta_{x(B_\ep^n)}$ with $z_\ep^n\in\Z$ for every $n=1,\ldots,N_\ep$\,. 
	Assume that
	\begin{equation*}
		\sup_{\ep>0}\frac{F(\B_\ep,\mu_\ep,V)}{|\log\rad(\B_\ep)|}\le C\,,
	\end{equation*} 
	for some constant $C>0$ independent of $\ep$\,.
	Then, the following facts hold true.
	\begin{itemize}
		\item[(i)] If $\widetilde\mu_\ep$ are the measures defined in \eqref{tilde} starting from the family $\Ccal_\ep(1)=\{B\in\B_\ep(1)\,:\,\overline{B}\subset V\}$, then $|\widetilde\mu_\ep|(V)\le C|\log\ep|$\,, and, up to a subsequence, $\widetilde\mu_\ep\res V\fla\mu$ (as $\ep\to 0$) for some $\mu\in X(V)$\,.
		\item[(ii)] $\pi|\mu|(V)\le \liminf_{\ep\to 0}\frac{F(\B_\ep,\mu_\ep,V)}{|\log\rad(\B_\ep)|}$\,.
	\end{itemize}
\end{theorem}
\begin{proof}
	See \cite[Theorem 3.2]{AP}.
\end{proof}
%
%
We are now in a position to prove Theorem \ref{mainthm}. The more involved part is the compactness property. For this reason we have splitted the argument in several steps.

\begin{proof}[Proof of Theorem \ref{mainthm}]
	{\it Proof of (i).} We start by covering with balls the jump set of $u_\ep$ and define the measure $\widetilde \mu_\ep$ in Step 1. In Steps 2 and 3 we suitably modify these balls in order to show - in Steps 4, 5, 6 and 7 - that the obtained measures $\pi\widetilde \mu_\ep$ are close to $Ju_\ep$ with respect to the flat distance. 
	Throughout the proof, the symbol $C$ denotes an absolute positive constant independent of the parameters, which might change from line to line.
	\medskip
	
	\textbf{Step 1: Construction of the starting family of balls.}
	By the energy bound \eqref{enbound}, we have that
	\begin{equation}\label{boundedjump}
		\Huno(\overline{S}_{u_\ep})\le C\ep|\log\ep|\,,
	\end{equation}
for all $\ep>0$.  By the very definition of Hausdorff measure, since $\overline{S}_{u_\ep}$ is compact, there exists a finite family of open balls $\B_\ep$ (in $\R^2$) such that $\overline{S}_{u_\ep}\subset\bigcup_{B\in\B_\ep}B$ and
	\begin{equation}\label{sumradii}
		\rad(\B_\ep)\le C\ep|\log\ep|\,,
	\end{equation}
	for some $C>0$\,. The number of balls in $\B_\ep$ depends on $\ep$. Moreover we notice that 
	\begin{equation}\label{huno}
		u_\ep\in H^1(\Omega(\B_\ep);\Ss^1)\,,
	\end{equation}
where we recall $\Omega(\B_\ep):=\Om\setminus (\cup_{B\in \B_\ep}\overline B)$.
	By \eqref{sumradii} and recalling that $\overline S_\ep\subset \overline{\Om'}$, we can assume, for $\ep$ small enough, that all the balls in $\B_\ep$ are contained in $\Om$\,.
	Up to applying a merging procedure (as in Definition \ref{merging_def}) for the balls in $\B_\ep$, we can assume without loss of generality that these balls are mutually (essentially) disjoint, and still satisfy \eqref{sumradii}.
	By \eqref{huno}, \eqref{sumradii} and \eqref{enbound}, for $\ep$ small enough it holds 
	\begin{equation}\label{nuovobound1}
		F({\B}_\ep,\mu_\ep,\Omega)\le\F_\ep(u_\ep)\le C|\log\ep|\le C|\log\rad({\B}_\ep)|\,,
	\end{equation}
	where $F$ is defined in \eqref{def:coreradius}.
	Using Proposition \ref{ballconstr}, we set 
	\begin{equation}\label{ccal}
		\Ccal_\ep:=\{B\in\B_\ep(1)\,:\, \overline B\subset\Omega\}
	\end{equation}
	and, according to \eqref{tilde}, we define
	\begin{equation}\label{tilde2}
		\widetilde\mu_\ep:=\sum_{B\in\Ccal_\ep}\deg(u_\ep,\partial B)\delta_{x(B)}\,.
	\end{equation}
	Notice that,  by definition of $\mathcal{AD}(\Omega,\Omega')$\,, actually $\Ccal_\ep\equiv \B_\ep(1)$ for $\ep$ small enough (i.e., all the closures of the balls in $\B_\ep(1)$ are contained in $\Om$).
	Therefore, by \eqref{nuovobound1}, we can apply Theorem \ref{alipons}(i) with $V=\Omega$, deducing that 
	\begin{equation}\label{boundmisura}
		|\widetilde\mu_\ep|({\Omega})\leq C|\log\ep|, 
	\end{equation}
	and, up to a subsequence, $\widetilde{\mu}_\ep\fla{\mu}$ (as $\ep\to 0$) for some ${\mu}\in X({\Omega})$\,.
	In order to conclude the proof of Theorem \ref{mainthm}(i) it is enough to show that 
	\begin{equation}\label{flatcomp}
		\|Ju_\ep-\pi\widetilde{\mu}_\ep\|_{\flt,\Omega}\to 0\qquad\textrm{ as }\ep\to 0\,.
	\end{equation}
	
	\medskip
	\textbf{Step 2: Fattening of the non-zero average clusters.} 
	Let $\Ccal_\ep$  be the family of balls in \eqref{ccal}. 
	Since $\{u_\ep\}_\ep\subset\mathcal{AD}(\Omega,\Omega')$ and in view of \eqref{sumradii}\,, for some open set $\Omega''$ with $\Omega'\subset\subset \Omega''\subset\subset\Omega$\,, we have that,
	for $\ep$ small enough, $B\subset \Omega''$ for every $B\in\Ccal_\ep$\,.
	We set 
	\begin{equation*}
	\nonnulla:=\{B\in \Ccal_\ep:\deg(u_\ep,\partial B)\neq 0\}\,,
	\end{equation*}
	and 
	\begin{equation*}
	\nulla:=\{B\in \Ccal_\ep:\deg(u_\ep,\partial B)= 0\}\,.
	\end{equation*}
	For every $B\in\nonnulla$ we replace the ball $B$ by $B^{\mathrm{mod}}:=B_{\ep\vee r(B)}(x(B))$ (hence increasing the radius up to $\ep$ if it is smaller). 
	Again, since the balls
	in $\widetilde{\Ccal}_\ep:=\nulla\cup\{B^{\mathrm{mod}}\,:\,B\in\nonnulla\}$
	 are not necessarily mutually disjoint, we pass to a merging procedure described in Definition \ref{merging_def}.

	By \eqref{sumradii} and  \eqref{boundmisura}, we still have
	\begin{align}\label{sumradii2}
		\rad(\widetilde{\Ccal}_\ep)\le C\ep|\log\ep|\,,
	\end{align}
	where $\widetilde{\Ccal}_\ep$ is now the new family of balls (not-relabelled after the merging procedure).
	Setting 
	\begin{equation*}
	\widetilde{\Ccal}_\ep^{\neq 0}:=\{B\in\widetilde{\Ccal}_\ep\,:\,\deg(u_\ep,\partial B)\neq 0\}\,,
	\end{equation*}
	by construction and by \eqref{boundmisura}, we have that
	\begin{equation*}
	\sharp\widetilde{\Ccal}_\ep^{\neq 0}\le \sharp\nonnulla\le C|\log\ep|\,.
	\end{equation*}
Once again, due to \eqref{sumradii2} and the fact that the original balls in $\Ccal_\ep$ have centers in $\Om'$, we can assume that the balls in $\widetilde\Ccal_\ep$ are all contained in $\Om''$.
	\medskip
	
	\textbf{Step 3: Dipoles elimination procedure. } 
	In this step we construct a family $\itezero$ of pairwise disjoint balls with
	\begin{equation}\label{sommaraggi}
	\sum_{B\in\itezero}r(B)\le 5\rad(\Ccal_\ep)\le\overline C\ep|\log\ep|,
	\end{equation}
	(with $\overline C>0$ a fixed constant) and
	\begin{equation}\label{conten}
	\bigcup_{B\in\Ccal_\ep}B\subset\bigcup_{B\in\itezero}B\,,
	\end{equation}
	such that
	{for any $B\in\itezero$ at least one of the following conditions is satisfied}
	\begin{equation}\label{bla}
	\textrm{(i) } r(B)\ge \frac{\ep}{2}\qquad\qquad\qquad\textrm{(ii)}\int_{\partial B}|\nabla u_\ep|\ud\Huno\le C|\log\ep|^{\frac{1}{2}}\,,
	\end{equation}
	for some universal constant $C>0$ (independent of $\ep$).
	
\medskip
	We start by classifying the balls in $\Ccal_\ep$ by setting
	\begin{align*}
		\mino:=\Big\{B\in \Ccal_\ep\,:\, r(B)<\frac\ep 2\Big\}\quad\textrm{and}\quad \magg:  =\Big\{B\in \Ccal_\ep\,:\, r(B)\ge\frac\ep 2\Big\}\,.
	\end{align*}
Notice that, if $\mino=\varnothing$\,, then,  we can set $\itezero:=\Ccal_\ep$ and, by \eqref{sumradii2}, the claim immediately follows.
	
If this is not the case, then  we adopt the iterative procedure described below.
For every $k=0,1,\ldots$, we will construct a pair of family of balls $(\Ccal_\ep(k);\itezero(k))$ and we set $$\ite(k):=\Ccal_\ep(k)\cup\itezero(k).$$
We classify the balls in $\Ccal_\ep(k)$ into two subclasses
	\begin{align*}
		\mino(k):=\Big\{B\in \Ccal_\ep(k)\,:\, r(B)<\frac\ep 2\Big\}\quad\textrm{and}\quad \magg(k):  =\Big\{B\in \Ccal_\ep(k)\,:\, r(B)\ge\frac\ep 2\Big\}\,,
	\end{align*}
and the balls in $\itezero(k)$ into two further subclasses
	\begin{align*}
		\iteuno(k):=\Big\{B\in \itezero(k)\,:\, r(B)<\frac\ep 2\Big\}\quad\textrm{and}\quad \itedue(k):  =\Big\{B\in \Ccal_\ep(k)\,:\, r(B)\ge\frac\ep 2\Big\}\,.
	\end{align*}
We initialize such families by setting $\Ccal_\ep(0):=\Ccal_\ep$\,, $\mino(0):=\mino$\,, $\magg(0):=\magg$\,, and $\itezero(0)=\iteuno(0)=\itedue(0)=\varnothing$\,.
We will now construct recursively a pair of families   $(\Ccal_\ep(k);\itezero(k))$ so that for a finite number $K_\ep\in \N$ we set $$\itezero:=\itezero(K_\ep).$$
The pair 
$(\Ccal_\ep(k);\itezero(k))$ will be such that 
for every $k=0,1,\dots, K_\ep-1$, $\Ccal_\ep(k)\supset \Ccal_\ep(k+1)$, $\itezero(k)\subset\itezero(k+1)$; moreover for all $k=0,1,\dots, K_\ep,$
\begin{equation}\label{disjo}
\bigcup_{B\in\Ccal_\ep(k)}B\cap\bigcup_{B\in\itezero(k)}B=\varnothing\,,
\end{equation}
\begin{equation}\label{contenute}
\Ccal_\ep(k)\subset\Ccal_\ep(0)\,,
\end{equation}
\begin{equation}\label{sommaraggiite}
\rad(\ite(k))\le \rad(\Ccal_\ep(0))+4\sum_{\widetilde B\in\itezero(k)}\sum_{\newatop{B\in\Ccal_\ep(0)}{B\subset\widetilde{B}}}r(B)\,,
\end{equation}
and {for any }$B\in\itezero(k)$
\begin{equation}\label{blaite}
\textrm{either }r(B)\ge \frac{\ep}{2}\quad\textrm{or}\quad\int_{\partial B}|\nabla u_\ep|\ud\Huno\le C|\log\ep|^{\frac{1}{2}}\,.
\end{equation}
We remark that \eqref{conten} now reads as
$$\bigcup_{B\in\Ccal_\ep(0)}B\subset\bigcup_{B\in\itezero(K_\ep)}B\,.$$
At the end of the procedure we will also have $$\Ccal_\ep(K_\ep)=\varnothing.$$
Eventually, we observe that \eqref{blaite} is equivalent to require
\begin{equation*}
	\int_{\partial B}|\nabla u_\ep|\ud\Huno\le C|\log\ep|^{\frac{1}{2}}\quad\textrm{for any }B\in\iteuno(k)\,.
\end{equation*}

	
	The pair $(\Ccal_\ep(0);\itezero(0))$ satisfies \eqref{disjo}, \eqref{contenute}, \eqref{sommaraggiite} and \eqref{blaite}.
	Let $k=0,1,\ldots$ and assume that  $(\Ccal_\ep(k);\itezero(k))$ satisfies \eqref{disjo}, \eqref{contenute}, \eqref{sommaraggiite} and \eqref{blaite}.
	\medskip
	
	If $\mino(k)=\varnothing$\,, then setting $k=:K_\ep$ we conclude the procedure according to  
	{formula \eqref{conclusionformula} below.}
	Otherwise, let $B_r(x)\in\mino(k)$\,.
	We set 
	\begin{equation*}
	\begin{aligned}
	\hatmino(k):=\mino(k)\setminus \{B_r(x)\}\,;\qquad\qquad&
	\hatmagg(k):=\magg(k)\,;\\
	\hatiteuno(k):=\iteuno(k)\,;\qquad\qquad&
	\widehat{\mathscr{I}}_\ep^{>}(k):=\itedue(k)\,;\\
	\widehat{\Ccal}_\ep(k):=\hatmino(k)\cup\hatmagg(k)\,;\qquad\qquad&\hatitezero(k):=\hatiteuno(k)\cup \hatitedue(k)\,.
	\end{aligned}
	\end{equation*}
	Finally, it is convenient to introduce $$\hatite(k):=\widehat{\Ccal}_\ep(k)\cup\hatitezero(k).$$
	Trivially, $(\widehat{\Ccal}_\ep(k)\cup\{B_r(x)\};\hatitezero(k))\equiv({\Ccal}_\ep(k);\itezero(k))$\,. Furthermore, in view of the inductive assumption,
	\begin{equation}\label{hatdisjo}
B_r(x)\cap\bigcup_{B\in\widehat{\Ccal}_\ep(k)}B\cap\bigcup_{B\in\hatitezero(k)}B=\varnothing\,,
\end{equation}
\begin{equation}\label{hatcontenute}
\widehat{\Ccal}_\ep(k)
\subset\Ccal_\ep(0)\,,
\end{equation}
\begin{equation}\label{hatsommaraggiite}
\rad(\hatite(k))+r\le \rad(\Ccal_\ep(0))+4\sum_{\widetilde B\in\hatitezero(k){\cup \{B_r(x)\}}}\sum_{\newatop{B\in\Ccal_\ep(0)}{B\subset\widetilde{B}}}r(B)\,,
\end{equation}
and for all $B\in\hatitezero(k)$ we have that
\begin{equation}\label{hatblaite}
\textrm{either }r(B)\ge \frac{\ep}{2}\quad\textrm{or}\quad\int_{\partial B}|\nabla u_\ep|\ud\Huno\le C|\log\ep|^{\frac{1}{2}}\,.
\end{equation}
{Notice that \eqref{sommaraggiite} implies \eqref{hatsommaraggiite}}.
{Moreover, we observe that \eqref{hatsommaraggiite} is equivalent to say that
\begin{equation*}
\rad(\hatite(k))+r\le \rad(\Ccal_\ep(0))+4\sum_{\widetilde B\in\hatitezero(k)}\sum_{\newatop{B\in\Ccal_\ep(0)}{B\subset\widetilde{B}}}r(B)+4r\,,
\end{equation*}
{in the special case that $B_r(x)\in \Ccal_\ep(0)$ (that is the case, for instance, if $k=0$).}
%
}

	 We now define
		\begin{equation}\label{defTrx}
		T(r,x):=\inf\Big\{t\geq 2r: \mathcal H^1(A_{B_r(x)}(t))\geq \frac{t-r}{2}\Big\}\,,
		\end{equation}
		where 
		\begin{align*}
			A_{B_r(x)}(t):=\{s\in[r,t]:\partial B_s(x)\cap \overline B=\varnothing,\; \forall B\in {\hatite(k)}\}\,,
		\end{align*}
		for all $t\geq 2r$\,. Notice that, being the map $t\mapsto \mathcal H^1(A_{B_r(x)}(t))$ continuous, the infimum in \eqref{defTrx} is actually a minimum.
	Furthermore, we set
		\begin{equation*}
		{d(x)}:=\textrm{dist}\Big(x,\bigcup_{B\in \hatitezero(k)}B\Big)\,,
		\end{equation*}
		{with the standard convention that $d(x)=+\infty$ if $\hatitezero(k)=\varnothing$\,.}
		
	\textit{Case 1: $T(r,x)\leq d(x)$\,.} In such a case $B_{T(r,x)}(x)$ does not intersect any ball in $\hatitezero(k)$\,. We distinguish two subcases.
	
	\vskip1mm
	\textit{Subcase 1a: $T(r,x)<\frac\ep2$\,.}
	In such a case, the ball $B_{T(r,x)}(x)$ does not intersect any ball in $\hatmagg(k)$\,, since balls in $\hatmagg(k)$ have radius  at least $\frac\ep2$\,.
		 On the other hand, it might contain some balls in $\hatmino(k)$ {and it contains for sure $B_r(x)$}\,.
		
		We choose a number $R=R(r,x)\in A_{B_r(x)}(T(r,x))$ such that 
		\begin{align}\label{stuno}
			\int_{\partial B_{R(r,x)}(x)}|\nabla u_\ep|\ud\mathcal H^1\leq\frac{1}{\mathcal H^1(A_{B_r(x)}(T(r,x)))} \int_{A_{B_r(x)}(T(r,x))}\int_{\partial B_s(x)}|\nabla u_\ep|\ud\mathcal H^1\ud s.
		\end{align}
		Denoting by $A\subset B_{T(r,x)}(x)$ the set (in polar coordinates centered at $x$)
		$$A:=\{(\rho,\theta):\rho\in A_{B_r(x)}(T(r,x))\}$$
		and using that $\mathcal H^1(A_{B_r(x)}(T(r,x)))=\frac{T(r,x)-r}{2}$\,, by \eqref{stuno}, we estimate
		\begin{align}\label{stdue}
			\int_{\partial B_{R(r,x)}(x)}|\nabla u_\ep|\ud\mathcal H^1&\leq \frac{2}{T(r,x)-r}\int_A|\nabla u_\ep|\ud x\leq \frac{2|A|^{\frac12}}{T(r,x)-r}\Big(\int_{A}|\nabla u_\ep|^2\ud x\Big)^{\frac12}\nonumber\\
			&\leq C\Big(\int_{A}|\nabla u_\ep|^2\ud x\Big)^{\frac12}\leq C|\log\ep|^{\frac12},
		\end{align}
		where $C$ is an absolute constant.  Since $R(r,x)<T(r,x)$, by definition of $T(r,x)$, we have
		\begin{align}\label{est_R}
			2\sum_{\substack{B\in \widehat{\Ccal}_\ep(k)\\B\subset B_{R(r,x)}(x)}} 
		r(B)\geq \mathcal H^1\big([r,R(r,x)]\setminus A_{B_r(x)}(R(r,x))\big)>\frac{R(r,x)-r}{2}.
		\end{align}
		Then, we set 
		\begin{equation*}
		\begin{aligned}
		\mino(k+1):=&\hatmino(k)\setminus\{B\in\hatmino(k)\,:\,B\subset B_{R(r,x)}(x)\}\,,\\
		\magg(k+1):=&\hatmagg(k)\,,\\
		\iteuno(k+1):=&\hatiteuno(k)\cup\{B_{R(r,x)}(x)\}\,,\\
		\itedue(k+1):=&\hatitedue(k)\,,\\
		\Ccal_\ep(k+1):=&\mino(k+1)\cup\magg(k+1)\,,\\
		\itezero(k+1):=&\iteuno(k+1)\cup \itedue(k+1)\,,\\
		\ite(k+1):=&\Ccal_\ep(k+1)\cup\itezero(k+1)\,.
		\end{aligned}
		\end{equation*}
		Notice that  the pair $(\Ccal_\ep(k+1);\itezero(k+1))$ satisfies \eqref{disjo}, by construction, \eqref{contenute}, by the inductive assumption, and \eqref{blaite}, by \eqref{stdue}.
		Finally, by {\eqref{hatdisjo}}, \eqref{hatcontenute}, \eqref{hatsommaraggiite}, and \eqref{est_R}, {using the very definition of $\iteuno(k+1)$\,,}
		we have that 
		\begin{equation*}
		\begin{aligned}
		\rad(\ite(k+1))=&\rad(\hatite(k))+r+R(r,x)-r-\sum_{\substack{B\in \widehat{\Ccal}_\ep(k)\\B\subset B_{R(r,x)}(x)}}r(B)\\
		\le&\rad(\Ccal_\ep(0))+4\sum_{\widetilde{B}\in\hatitezero(k){\cup\{B_r(x)\}}}\sum_{\substack{B\in \Ccal_\ep(0)\\B\subset\widetilde{B}}}r(B)+4\sum_{\substack{B\in \widehat{\Ccal}_\ep(k)\\B\subset B_{R(r,x)}(x)}}r(B)\\
		=&\rad(\Ccal_\ep(0))+4\sum_{\widetilde{B}\in\hatitezero(k)}\sum_{\substack{B\in \Ccal_\ep(0)\\B\subset\widetilde{B}}}r(B)+{ 4\sum_{\substack{B\in\Ccal_\ep(0)\\B\subset B_r(x)}}r(B)}+4\sum_{\substack{B\in \widehat{\Ccal}_\ep(k)\\B\subset B_{R(r,x)}(x)}}r(B)\,,\\
		\le&{\rad(\Ccal_\ep(0))+4\sum_{\widetilde{B}\in\hatitezero(k)}\sum_{\substack{B\in \Ccal_\ep(0)\\B\subset\widetilde{B}}}r(B)+4\sum_{\substack{B\in {\Ccal}_\ep(0)\\B\subset B_{R(r,x)}(x)}}r(B)}\\
		\le&\rad(\Ccal_\ep(0))+4\sum_{\widetilde{B}\in\itezero(k+1)}\sum_{\substack{B\in \Ccal_\ep(0)\\B\subset \widetilde B}}r(B)\,,
		\end{aligned}
		\end{equation*}
		which implies \eqref{sommaraggiite}.
		\vskip2mm

{\it Subcase 1b: $T(r,x)\ge\frac\ep 2$\,.} In this case, we define $R(r,x):=T(r,x)$ and we set
		\begin{equation*}
		\begin{aligned}
		\mino(k+1):=&\mino(k)\setminus\{B\in\mino(k)\,:\,B\subset B_{R(r,x)}(x)\}\,,\\
		\magg(k+1):=&\magg(k)\setminus\{B\in\magg(k)\,:\,B\subset B_{R(r,x)}(x)\}\,,\\
		\iteuno(k+1):=&\iteuno(k)\,,\\
		\itedue(k+1):=&\itedue(k)\cup\{B_{R(r,x)}(x)\}\,,\\
			\Ccal_\ep(k+1):=&\mino(k+1)\cup\magg(k+1)\,,\\
		\itezero(k+1):=&\iteuno(k+1)\cup \itedue(k+1)\,.
		\end{aligned}
		\end{equation*}
	{By arguing as in Subcase 1a also in this case we obtain that the pair $(\Ccal_\ep(k+1);\itezero(k+1))$ satisfies \eqref{disjo}, \eqref{contenute},  \eqref{sommaraggiite} and \eqref{blaite}.}
\medskip
		
	{\it Case 2: $T(r,x)> d(x)$\,.} In this case, we set {$$R(r,x):=\sup\{t\le d(x)\,:\,\partial B_t(x)\cap B=\varnothing\textrm{ for all }B\in\hatite(k)\}$$} and
		\begin{equation*}
		\widetilde{\ite}(k):=\widetilde{\Ccal}_\ep^{<}(k)\cup\widetilde{\Ccal}_\ep^{>}(k)\cup \widetilde{\mathscr{I}}_\ep^<(k)\cup \widetilde{\mathscr{I}}_\ep^>(k)\,,
		\end{equation*}
		where 
		\begin{equation}\label{instildepezzi}
		\begin{aligned}
		\widetilde{\Ccal}_\ep^{<}(k):=&\hatmino(k)\setminus\{B\in\hatmino(k)\,:\, B\subset B_{R(r,x)}(x)\}\,,\\
		\widetilde{\Ccal}_\ep^{>}(k):=&\hatmagg(k)\setminus\{B\in\hatmagg(k)\,:\, B\subset B_{R(r,x)}(x)\}\,,\\
		\widetilde{\mathscr{I}}_\ep^<(k):=&\hatiteuno(k)\,,\\
		\widetilde{\mathscr{I}}_\ep^>(k):=&\hatitedue(k)\,.
		\end{aligned}
		\end{equation}
{Notice that $\partial B_{R(r,x)}(x)$ does not  intersect other balls in $\widetilde{\ite}(k)$, but it might contains balls in $\widehat{\Ccal}_\ep(k)$ and it can be tangent to some ball in $\widetilde{\mathscr{I}}_\ep(k)$ (if it happens that $R(r,x)=d(x)$). Furthermore, we notice that, by the very definition of $R(r,x)$, $B_{R(r,x)}(x)$ is surely tangent to some ball in $\hatite(k)$\,. Then, we pass through a merging procedure (Definition \ref{merging_def}), including in a unique ball $B_{R'}(x')$ all the balls in $\widetilde{\ite}(k)$ 
	whose closures intersect $\partial B_{R(r,x)}(x)$ (and possibly others).  We emphasize that, by definition of $R(r,x)$, after the merging procedure $B_{R'}(x')$ must contain some other ball in $\widetilde \itezero(k)$.}
Then, we redefine the sets in \eqref{instildepezzi} by setting
\begin{equation*}
		\begin{aligned}
		\widetilde{\Ccal}_\ep^{<}(k):=&\hatmino(k)\setminus\{B\in\hatmino(k)\,:\, B\subset B_{R'}(x')\}\,,\\
		\widetilde{\Ccal}_\ep^{>}(k):=&\hatmagg(k)\setminus\{B\in\hatmagg(k)\,:\, B\subset B_{R'}(x')\}\,,\\
		\widetilde{\mathscr{I}}_\ep^<(k):=&\hatiteuno(k)\setminus\{B\in\hatiteuno(k)\,:\, B\subset B_{R'}(x')\}\,,\\
		\widetilde{\mathscr{I}}_\ep^>(k):=&\hatitedue(k)\setminus\{B\in\hatitedue(k)\,:\, B\subset B_{R'}(x')\}\,;
		\end{aligned}
		\end{equation*}
		moreover, we set 
		\begin{equation*}
\widetilde{\Ccal}_\ep(k):=\widetilde{\Ccal}_\ep^{<}(k)\cup \widetilde{\Ccal}_\ep^{>}(k)\,,\qquad \widetilde{\mathscr{I}}_\ep(k):=\widetilde{\mathscr{I}}_\ep^<(k)\cup \widetilde{\mathscr{I}}_\ep^>(k)\,,\qquad \widetilde{\mathscr{S}}_\ep(k):=\widetilde{\Ccal}_\ep(k)\cup\widetilde{\mathscr{I}}_\ep(k)\,.
\end{equation*}
		We notice that, by definition of $R(r,x)$ and by \eqref{est_R}, we have
		\begin{equation*}
		R(r,x)-r<
	4\sum_{\substack{B\in \widehat{\Ccal}_\ep(k)\\B\subset B_{R(r,x)}(x)}}r(B)\,,
		\end{equation*}
{whence, by the very definition of merging, we deduce
		\begin{equation}\label{nuovasommaraggiite}
	\begin{aligned}
	\rad(\widetilde{\mathscr{S}}_\ep(k))+R'
	\le&	\rad(\widetilde{\mathscr{S}}_\ep(k))+\sum_{\newatop{B\in\hatite(k)}{B\subset(B_{R'}(x')\setminus B_{R(r,x)}(x))}}r(B)+R(r,x)\\
\le&\rad(\widetilde{\mathscr{S}}_\ep(k))+\sum_{\newatop{B\in\hatite(k)}{B\subset(B_{R'}(x')\setminus B_{R(r,x)}(x))}}r(B)+4\sum_{\substack{B\in\widehat{\Ccal}_\ep(k)\\B\subset B_{R(r,x)}(x)}}r(B)+r\\
	\le&\rad(\hatite(k)){+r}+4\sum_{\substack{B\in\widehat{\Ccal}_\ep(k)\\B\subset B_{R(r,x)}(x)}}r(B)\\
	\le&\rad(\Ccal_\ep(0))+4\sum_{\widetilde B\in\hatitezero(k)\cup\{B_r(x)\}}\sum_{\newatop{B\in\Ccal_\ep(0)}{B\subset\widetilde{B}}}r(B)+4\sum_{\substack{B\in\widehat{\Ccal}_\ep(k)\\B\subset B_{R(r,x)}(x)}}r(B)
	\\
	\leq&\rad(\Ccal_\ep(0))+4\sum_{\widetilde B\in\widetilde\itezero(k)}\sum_{\newatop{B\in\Ccal_\ep(0)}{B\subset\widetilde{B}}}r(B)+4\sum_{\newatop{B\in\Ccal_\ep(0)}{B\subset B_r(x)}}r(B)\\
	&+4\sum_{\newatop{\widetilde B\in\widehat\itezero(k)}{\widetilde B\subset B_{R'}(x')}}\sum_{\newatop{B\in\Ccal_\ep(0)}{B\subset\widetilde{B}}}r(B)+4\sum_{\substack{B\in\widehat{\Ccal}_\ep(k)\\B\subset B_{R(r,x)}(x)}}r(B)\\
	\leq& \rad(\Ccal_\ep(0))+4\sum_{\widetilde B\in\widetilde\itezero(k)\cup\{B_{R'}(x')\}}\sum_{\newatop{B\in\Ccal_\ep(0)}{B\subset\widetilde{B}}}r(B)\,,
\end{aligned}
	\end{equation}
where we have used also \eqref{hatsommaraggiite} and \eqref{hatdisjo}.
Now, if $R'\ge\frac\ep 2$\,, we set
\begin{equation*}
	\begin{aligned}
		\mino(k+1)&:=\widetilde{\Ccal}_\ep^{<}(k),\\
		\magg(k+1)&:=\widetilde{\Ccal}_\ep^{>}(k),\\
		\iteuno(k+1)&:=\widetilde{\mathscr{I}}_\ep^<(k),\\
		\itedue(k+1)&:=
		\widetilde{\mathscr{I}}_\ep^>(k)\cup\{B_{R'}(x')\},\\
			\Ccal_\ep(k+1)&:=\mino(k+1)\cup\magg(k+1),\\
		\itezero(k+1)&:=\iteuno(k+1)\cup \itedue(k+1),
	\end{aligned}
\end{equation*}
and we conclude the iterative step. Notice that the family $$\ite(k+1)=\mino(k+1)\cup \magg(k+1)\cup \iteuno(k+1)\cup \itedue(k+1)$$ is made of pairwise disjoint balls, and  by \eqref{nuovasommaraggiite} it satisfies \eqref{hatdisjo}, \eqref{hatcontenute}, \eqref{hatsommaraggiite}, \eqref{hatblaite}. }

Otherwise, if $R'<\frac{\ep}{2}$\,, we redefine $B_r(x)$ by setting {$B_r(x):=B_{R'}(x')$\,, and relabelling
\begin{equation*}
		\begin{aligned}
		\hatmino(k):=\widetilde{\Ccal}_\ep^{<}(k),\qquad\qquad&\hatmagg(k):=\widetilde{\Ccal}_\ep^{>}(k)\\
		\hatiteuno(k):=\widetilde{\mathscr{I}}_\ep^<(k)\,,\qquad\qquad&\hatitedue(k):=
		\widetilde{\mathscr{I}}_\ep^>(k)\,,\\
		\widehat{\Ccal}_\ep(k):=	\hatmino(k)\cup\hatmagg(k),\qquad\qquad&\hatitezero(k):=\hatiteuno(k)\cup\hatitedue(k)\,,
		\end{aligned}
		\end{equation*}
	and $\hatite(k):=\widehat{\Ccal}_\ep(k)\cup\hatitezero(k)$\,; we check that (the new ball) $B_r(x)$ and $\hatite(k)$ satisfy \eqref{hatdisjo}, \eqref{hatcontenute}, \eqref{hatsommaraggiite}, \eqref{hatblaite}, and we restart the whole process by defining $T(r,x)$ as in \eqref{defTrx} starting from the new (just built) ball $B_r(x)$\,.
}

{If we fall again in Case 2 and the new radius $R'<\frac\ep2$, we iterate again the procedure. Notice that every time we fall in this situation, the number of balls in $\hatitezero(k)$ decreases (due to the merging procedure), and hence after a finite number $N_k>0$ of iterations we will end up in some other case. In other words, after a finite  number of steps (depending on $\ep$), such an iteration cannot be applied anymore and hence we necessarily fall either in case $R'\geq\frac\ep2$ or in Case 1, thus defining $\hatite(k+1)$ accordingly.
}
\medskip

		\textit{(Conclusion of the procedure)} We iterate the process described above inductively. Every time we perform this iteration the number of balls in $\mino$ decreases, and after a finite number $K_\ep$ of processes we end up with $\mino(K_\ep)=\varnothing$\,.  Finally, we set 
		\begin{equation}\label{conclusionformula}
		\iteuno:=\iteuno(K_\ep)\qquad\qquad\itedue:=\itedue(K_\ep)\cup\magg(K_\ep)\,,\qquad\qquad\itezero:=\iteuno\cup\itedue\,.
		\end{equation}
		By construction (see \eqref{disjo}-\eqref{blaite}), the balls in $\itezero$ are pairwise disjoint and satisfy \eqref{sommaraggi}-\eqref{bla}.
		Furthermore, we observe that
		\begin{equation}\label{quasi0}
			\int_{\partial B}|\nabla u_\ep|\ud\Huno\le C|\log\ep|^{\frac 1 2}\qquad\qquad\textrm{for every }B\in\iteuno\,,
		\end{equation}
		and, in view of \eqref{sommaraggi}, 
		\begin{equation}\label{semprestessa}
			\sharp\itedue\le C|\log\ep|\,.
	\end{equation}
	\medskip

	\textbf{Step 4: Estimate of $\|\nabla u_\ep\|_{L^1}$ on the boundary of the balls in {$\itedue$}\,.}
	For every $B=B_r(x)\in \itedue$\,, we set
	\begin{equation}\label{nuovo}
		R(r,\ep):=16\,\overline C\,|\log\ep|\,r\,,
	\end{equation}
	where $\overline C$ is the constant appearing in \eqref{sommaraggi}. Notice that 
	\begin{equation}\label{newRrep}
	{\sum_{B_r(x)\in\itedue}R(r,\ep)\leq 16\overline C|\log\ep|\rad(\itezero)\leq C\ep|\log\ep|^2}\,,
	\end{equation} 
	and hence, for $\ep$ small enough, we can assume that all the balls $B_{R(r,\ep)}(x)$ are contained in $\Om$.
	Let $\widetilde D^1,\ldots,\widetilde D^{M_\ep}$ denote the connected components of the set $\bigcup_{B_r(x)\in\itedue}B_{R(r,\ep)}(x)$\,.
	By \eqref{semprestessa} we get
	\begin{align*}
		M_\ep\leq C|\log\ep|\,.
	\end{align*}
	For every $m=1,\ldots,M_\ep$\,, let $\{B_{r^{m,k}}(x^{m,k})\}_{k=1,\ldots,K^m}\subseteq\itedue$ be such that $\widetilde D^m:=\bigcup_{k=1}^{K^m}B_{R^{m,k}_\ep}(x^{m,k})$\,, where $R^{m,k}_\ep:=R(r^{m,k},\ep)$ is defined by \eqref{nuovo}\,. Moreover, for every $m=1,\ldots,M_\ep$\,, we set
	$\widehat D^m:=\bigcup_{k=1}^{K^m}B_{r^{m,k}}(x^{m,k})$ and we define 
	%
	the function $f^m:\widetilde{D}^m\to [0,1]$ as 
	$$
	f^m(x):=\inf_{k=1,\ldots,K^m}\Big\{\frac{1}{(16\,\overline C\,|\log\ep|-1)r^{m,k}} \di(x, B_{r^{m,k}}(x^{m,k}))\Big\}\,; 
	$$
	we notice that $f^m$ is Lipschitz continuous, $f^m=0 $ on $\widehat D^m$, $f^m=1$ on $\partial \widetilde D^m$, and 
	\begin{equation}\label{stimagradfm}
		|\nabla f^m|\le \frac{1}{(16\,\overline C\,|\log\ep|-1)\frac{\ep}{2}}\,.
	\end{equation} 
	Finally, for every $m=1,\ldots, M_\ep$\,, let us define the sets
	\begin{equation*}
		\begin{aligned}
			T^m_1:=&\Big\{t\in [0,1]\,:\,\{f^m=t\}\cap\bigcup_{B\in\iteuno} B\neq \varnothing\Big\}\,,\\
			T^m_2:=&\Big\{t\in [0,1]\,:\,\{f^m=t\}\cap\bigcup_{B\in\iteuno} B=\varnothing\Big\}\,,
		\end{aligned}
	\end{equation*}
	and we show that 
	\begin{equation}\label{stimaT2}
		\Huno(T^m_2)\ge \frac 1 2\,.
	\end{equation}
	Indeed, in view of \eqref{stimagradfm}) for every $B\in\iteuno$, 
	\begin{equation*}
		\begin{aligned}
			\Huno\Big(\left\{t\in [0,1]\,:\,\{f^m=t\}\cap B\neq\varnothing\right\}\Big)\leq&
			\mathrm{osc}_B(f^m)\le 2r(B)|\nabla f^m|\\
			\le&2r(B)\frac{1}{(16\,\overline C\,|\log\ep|-1)\frac{\ep}{2}}\le \frac{r(B)}{2\overline C\ep|\log\ep|}\,,
		\end{aligned}
	\end{equation*}
	which, together with \eqref{sommaraggi}, implies
	\begin{equation*}
		\Huno(T^m_1)\le \sum_{B\in\iteuno}\Huno\Big(\left\{t\in [0,1]\,:\,\{f^m=t\}\cap B\neq\varnothing\right\}\Big)\le\frac{\rad(\iteuno)}{2\overline C\ep|\log\ep|}\le\frac 1 2\,,
	\end{equation*}
	whence \eqref{stimaT2} follows.
	For every $k=1,\ldots,K^m$ we set 
	$$
	E^m_{r^{m,k}}:=\Big\{x\in\widetilde D^m\,:\,|\nabla f^m(x)|=\frac{1}{(16\,\overline C\,|\log\ep|-1)r^{m,k}}\Big\}\,;
	$$
	 it is easy to check that 
	 \begin{equation*}
		E^m_{r^{m,k}}\subset  B_{16\,\overline C\,|\log\ep|\, r^{m,k}}(x^{m,k})\setminus \overline{B}_{r^{m,k}}(x^{m,k})\subset \widetilde D^m\setminus \widehat D^m\,.
	\end{equation*} 
	
	Since $f^m$ is Lipschitz continuous, by the coarea formula and by the mean value theorem, there exists $t^m\in T^m_2$ such that
	\begin{equation}\label{quasi}
		\begin{aligned}
			&\int_{\{f^m(x)=t^m\}}|\nabla u_\ep|\ud\Huno\le \frac{1}{\Huno(T^m_2)}\int_{T^m_2}\ud t\int_{\{f^m(x)=t\}}|\nabla u_\ep|\ud\Huno\\
			\le&2\int_{(\widetilde{D}^m\setminus\widehat D^m)\setminus\bigcup_{B\in\iteuno}B}|\nabla u_\ep||\nabla f^m|\ud x
			=\sum_{k=1}^{K^m}\frac{1}{(16\,\overline C\,|\log\ep|-1)r^{m,k}}\int_{E^{m}_{r^{m,k}}}|\nabla u_\ep|\ud x\\
			\le&K^m\|\nabla u_\ep\|_{L^2(\Omega;\R^{2\times 2})}\le C|\log\ep|^{\frac 3 2}\,,
		\end{aligned}
	\end{equation}
	where the last but one inequality follows by H\"older inequality and in last passage we have used \eqref{semprestessa} and \eqref{enbound}\,.
	For every $m=1,\ldots,M_\ep$\,, we set
	\begin{equation*}
		D^m:=\{x\in \widetilde D^m\,:\,f^m(x)\le t^m\}\,,
	\end{equation*}
so that $\widehat D^m\subset D^m$;
	by \eqref{quasi} we have that
	\begin{equation}\label{quasi2}
		\int_{\partial D^m}|\nabla u_\ep|\ud\Huno\le C|\log\ep|^{\frac 3 2}\,. 
	\end{equation}
	Moreover, we notice that, {in view of \eqref{newRrep}\,,}
	\begin{equation}\label{sommadiametri}
		\sum_{m=1}^{M_\ep}\mathrm{diam}(D^m)\le C\ep|\log\ep|^2\,.
	\end{equation}
We remark that 
\begin{equation}\label{alphak}
	D^m:=\bigcup_{k=1}^{K^m}B_{\alpha^{m,k}r^{m,k}}(x^{m,k})\,,\quad\textrm{for some }\alpha^{m,k}\in(1,16\overline C|\log\ep|)\textrm{ (for all $k=1,\dots,K^m$)}\,;
\end{equation} 
by \eqref{newRrep},
\begin{equation}\label{sommaalfakappa}
\sum_{k=1}^{K^m}\alpha^{m,k}r^{m,k}\le C\ep|\log\ep|^2\,.
\end{equation}
Eventually, observe that, thanks to the choice of $t^m\in T^m_2$, the boundary $\partial D^m$ does not intersect other balls in $\itezero$\,. However, $D^m$ might contain some other ball in $\iteuno$\,.
	\medskip

	\textbf{Step 5: Definition of $\widehat\mu_\ep$\,.}
	For every $\ep>0$\,, we define $\widehat\mu_\ep$ as 
	\begin{equation*}
		\widehat\mu_\ep:=\sum_{m=1}^{M_\ep}\deg(u_\ep,\partial D^m)\delta_{\widehat x^m}\,,
	\end{equation*}
	where, for every $m=1,\ldots,M_\ep$\,, $\widehat x^m$ is the center of an arbitrarily chosen ball among the balls $\{B_{r^{m,k}}(x^{m,k})\}_{k=1,\ldots,K^m}$ contained in $D^m$\,.
	Notice that $\deg(u_\ep,\partial D^m)=\sum_{\newatop{B\in\Ccal_\ep}{B\subset D^m}}\deg(u_\ep,\partial B)$\,.
	 As a consequence, it easily follows by the very definition of $\widetilde\mu_\ep$ in \eqref{tilde2}, that 
	\begin{align}\label{stima_mucap}
	|\widehat \mu_\ep|(\Om)\leq |\widetilde\mu_\ep|(\Om)\leq C|\log\ep|\,.
	\end{align}
	By \eqref{sommadiametri} and Theorem \ref{alipons}(i),  we get
	\begin{equation*}
		\begin{aligned}
			\|\widetilde{\mu}_\ep-\widehat{\mu}_\ep\|_{\flt}\le&\sup_{\|\ffi\|_{\Cc^{0,1}(\Omega)}\le 1}\sum_{m=1}^{ M_\ep}\Big|\int_{D^m}\ffi\ud(\widetilde{\mu}_\ep-\widehat{\mu}_\ep)\Big|
			\le2\sup_{\|\ffi\|_{\Cc^{0,1}(\Omega)}\le 1}\sum_{m=1}^{{M}_\ep}|\widetilde\mu_\ep|(B)\mathrm{osc}_{D^m}(\ffi)\\
			\le& C\ep|\log\ep|^3\,,
		\end{aligned}
	\end{equation*}
and hence $\|\widetilde{\mu}_\ep-\widehat{\mu}_\ep\|_{\flt}\to0$ as $\ep\rightarrow0$\,. Thus,  in order to get \eqref{flatcomp}, we are left to prove that
%
		\begin{equation}\label{20220424_7}
			\|Ju_\ep-\pi\widehat\mu_\ep\|_{\flt,\Omega}\to 0\qquad\textrm{as }\ep\to 0\,.
		\end{equation}
		

\medskip

{\textbf{Step 6: Definition of $\lambda_\ep$ and $\widehat{\lambda}_\ep$\,.}}
	In order to prove \eqref{20220424_7} we introduce the measures $\lambda_\ep$ and $\widehat{\lambda}_\ep$ such that the following identities
	\begin{equation}\label{correnti}
-\mathrm{Div}\,\lambda_\ep=Ju_\ep\qquad\textrm{and}\qquad-\mathrm{Div}\,\widehat\lambda_\ep={\pi}\widehat\mu_\ep\,,
	\end{equation}
hold in the sense of distributions in $\mathcal D'(\Om)$.
	We set 
	\begin{equation*}
	\lambda^1_\ep:=\frac12([u_\ep^1\mad_1u_\ep^{2}]-[u_\ep^{2}\mad_1u_\ep^{1}])\,,\qquad \lambda^2_\ep:=\frac12([u_\ep^1\mad_2u_\ep^{2}]-[u_\ep^{2}\mad_2u_\ep^{1}])\,,\qquad
	\lambda_\ep:=(-\lambda^2_\ep,\lambda_\ep^1)\,.
	\end{equation*}
By definition of weak Jacobian determinant (see Definition \ref{def_2x2minor}, and \eqref{def_dTu_2dim}), $\lambda_\ep$ satisfies the first equation in \eqref{correnti}.
	Moreover, by \eqref{quasi0} and \eqref{quasi2},
	\begin{equation}\label{quasiquasi}
		\begin{aligned}
			\int_{\partial B}|\lambda_\ep|\ud \mathcal H^1\le C|\log\ep|^{\frac 1 2}&\qquad\textrm{for every }B\in\iteuno\\
			\int_{\partial D^m}|\lambda_\ep|\ud \mathcal H^1\le C|\log\ep|^{\frac 3 2}&\qquad\textrm{for every }m=1,\ldots,M_\ep\,.
		\end{aligned}
	\end{equation}
	Furthermore, by H\"older inequality,  \eqref{enbound}, and \eqref{boundedjump} (for the estimate of $|\mad^su_\ep|$), we get
		\begin{equation}\label{20220424_5}
			\begin{aligned}
				|\lambda_\ep|\bigg(\bigcup_{B\in\iteuno}B\cup\bigcup_{m=1}^{M_\ep}D^m\bigg)\le&|\mad u_\ep|\Big(\bigcup_{B\in\iteuno}B\cup\bigcup_{m=1}^{M_\ep}D^m\Big) \\
				\le& \|\nabla u_\ep\|_{L^2(\Omega;\R^{2\times 2})}\bigg(\sum_{B\in\iteuno}|B|+\sum_{m=1}^{M_\ep}|D^m|\bigg)^{\frac 1 2}+C\ep|\log\ep|\\		
				\le& C\ep|\log\ep|^{{\frac 5 2}}\,,
			\end{aligned}
		\end{equation}
	where we have used also \eqref{sommaraggi} and \eqref{sommadiametri} to deduce that
	 \begin{equation}\label{stimaaree}
	 \sum_{B\in{\iteuno}}|B|\le C\ep^2|\log\ep|^2\qquad\textrm{and}\qquad\sum_{m=1}^{M_\ep}|D^m|\leq C\ep^2|\log\ep|^4\,.
	 \end{equation}
	Now, for every $m=1,\dots,M_\ep$, we define $v^m_\ep\in W^{1,1}(\Om;\Ss^1)$
	as
		$$v^m_\ep(x):=e^{\imath\deg(u_\ep,\partial D^m)\vartheta(x-\widehat x^m)},$$
	where $\R^2$ is identified with $\C$ and
	$\vartheta$ is the angular polar coordinate defined by 
	\begin{equation}\label{deftheta0}
		\vartheta(x):=\left\{\begin{array}{ll}
			\arctan\frac{x_2}{x_1}&\textrm{if }x_1>0\\
			\frac{\pi}{2}&\textrm{if }x_1=0 \textrm{ and }x_2>0\\
			\pi+\arctan\frac{x_2}{x_1}&\textrm{if }x_1<0\\
			\frac{3}{2}\pi&\textrm{if }x_1=0  \textrm{ and }x_2<0\,.
		\end{array}
		\right.
	\end{equation}
	We set $$v_\ep(\cdot):=\Pi_{m=1}^{M_\ep}v_\ep^m(\cdot)\equiv e^{\imath\sum_{m=1}^{M_\ep}\deg(u_\ep,\partial D^m)\vartheta(\cdot-\widehat x^m)}$$ and we define
\begin{equation*}
\widehat \lambda^1_\ep:=\frac12(v_\ep^1\nabla_1v_\ep^{2}-v_\ep^{2}\nabla_1v_\ep^{1})\,,\qquad \widehat\lambda^2_\ep:=\frac12(v_\ep^1\nabla_2v_\ep^{2}-v_\ep^{2}\nabla_2v_\ep^{1})\,,\qquad
	\widehat \lambda_\ep:=(-\widehat \lambda_\ep^2,\widehat \lambda_\ep^1)\,.
\end{equation*}
	{We notice that $v_\ep\in W^{1,p}(\Omega;\Ss^1)$ for any $1\le p<2$\,.}
	It is well-known that, for every $m=1,\ldots,M_\ep$\,, the distributional determinant of $v^m_\ep$ coincides with ${\pi}\deg(u_\ep,\partial D^m)\delta_{\widehat x^m}={\pi}\widehat \mu_\ep\res D^m$\,, so that also the second equation in \eqref{correnti} holds true.
	Finally, we show that $\widehat{\lambda}_\ep$ satisfies estimates similar to \eqref{quasiquasi} and \eqref{20220424_5}.
	
	Let $m=1,\ldots,M_\ep$\,. Recalling \eqref{alphak}, we can assume without loss of generality that $\widehat x^m=x^{m,1}$\,.
	Then, by \eqref{alphak}, \eqref{sommaalfakappa}, \eqref{sommadiametri} and \eqref{stima_mucap}, we have
			\begin{align}\label{stima_lambdahat}
		\int_{\partial D^m}|\widehat\lambda_\ep|\ud\Huno\leq& \int_{\partial D^m} |\nabla v_\ep|\ud \Huno\leq \sum_{j=1}^{M_\ep}|\deg(u_\ep,\partial D^j)|\int_{\partial D^m}|\nabla \vartheta (\cdot-\widehat x^j)|\ud\Huno\nonumber\\
		\le&|\deg(u_\ep,\partial D^m)|\int_{\partial D^m}|\nabla \vartheta (\cdot-\widehat x^m)|\ud\Huno\nonumber\\
		&+\sum_{\newatop{j=1}{j\neq m}}^{M_\ep}|\deg(u_\ep,\partial D^j)|\int_{\partial D^m}|\nabla \vartheta (\cdot-\widehat x^j)|\ud\Huno\nonumber\\
		\le&|\deg(u_\ep,\partial D^m)|\int_{\partial B_{\alpha^{m,1}r^{m,1}}(x^{m,1})}|\nabla \vartheta (\cdot-x^{m,1})|\ud\Huno\nonumber\\
		&+|\deg(u_\ep,\partial D^m)|\sum_{k=2}^{K^m}\int_{\partial B_{\alpha^{m,k}r^{m,k}}(x^{m,k})\setminus B_{\alpha^{m,1}r^{m,1}}(x^{m,1})}|\nabla \vartheta (\cdot-x^{m,1})|\ud\Huno\nonumber
		\\
		&+\sum_{\newatop{j=1}{j\neq m}}^{M_\ep}|\deg(u_\ep,\partial D^j)|\frac{\Huno(\partial D^j)}{\ep}\nonumber\\
		\le&2\pi|\deg(u_\ep,\partial D^m)|+2\pi|\deg(u_\ep,\partial D^m)|\sum_{k=1}^{K^m}\frac{\alpha^{m,k}r^{m,k}}{\ep}\nonumber\\
		&+C\sum_{\newatop{j=1}{j\neq m}}^{M_\ep}|\deg(u_\ep,\partial D^j)|\frac{\mathrm{diam}(D^j)}{\ep}
		\nonumber\\
\le&C|\widehat\mu_\ep|(\Omega)+C|\widehat\mu_\ep|(\Omega)|\log\ep|^2\le C|\log\ep|^3\,.
		\end{align}
{Analogously, recalling that the balls in $\itezero$ are pairwise disjoint and that the balls in $\itedue$ have radius larger than $\frac \ep 2$\,, for every $B\in\iteuno$ we have
\begin{equation}\label{stima_lambdahat2}
\begin{aligned}
	\int_{\partial B}|\widehat{\lambda}_\ep|\ud \Huno\le&
	\int_{\partial B} |\nabla v_\ep|\ud \Huno\leq \sum_{j=1}^{M_\ep}|\deg(u_\ep,\partial D^j)|\int_{\partial B}|\nabla \vartheta (\cdot-\widehat x^j)|\ud\Huno\\
	\le&2\sum_{j=1}^{M_\ep}|\deg(u_\ep,\partial D^j)|\frac{\Huno(\partial B)}{\ep}\le C|\log\ep|\,,
	\end{aligned}
	\end{equation}
	where the last inequality follows by \eqref{sommaraggi} and \eqref{stima_mucap}.
	}

	Finally, we prove that
	\begin{equation}\label{20220424_6}
	|\widehat{\lambda}_\ep|\bigg(\bigcup_{B\in\iteuno}B\cup\bigcup_{m=1}^{M_\ep}D^m\bigg)\le C\ep^{\frac 2 3}|\log\ep|^{\frac 7 3}\,.
	\end{equation}
	By definition of $v_\ep$ we easily get, for some fixed $p\in (1,2)$\,, and thanks to \eqref{stima_mucap},
	\begin{equation*}
	\begin{aligned}
\|\nabla v_\ep\|_{L^p( \Om;\R^{2\times 2})}=&\Big\|\sum_{m=1}^{M_\ep}\deg(u_\ep,\partial D^m)\nabla\vartheta(\cdot-\widehat x^m) \Big\|_{L^p(\Om;\R^{2})}\\
\leq& C\sum_{m=1}^{M_\ep}|\deg(u_\ep,\partial D^m)|=C|\widehat \mu_\ep|(\Om)\leq C|\log\ep|\,.
	\end{aligned}
	\end{equation*}
	Hence,
	{ by H\"older inequality and \eqref{stimaaree}, we deduce
			\begin{equation*}
				\begin{aligned}
				|\widehat{\lambda}_\ep|\bigg(\bigcup_{B\in\iteuno}B\cup\bigcup_{m=1}^{M_\ep}D^m\bigg)\le& \int_{\bigcup_{B\in\iteuno}B\cup\bigcup_{m=1}^{M_\ep}D^m}|\nabla v_\ep|\ud x\\
				\le& \|\nabla v_\ep\|_{L^p(\Omega;\R^{2\times 2})}\bigg(\sum_{B\in\iteuno}|B|+\sum_{m=1}^{M_\ep}|D^m|\bigg)^{\frac{p-1}{p}}\\
				\le& C\ep^{\frac{2p-2}{p}}|\log\ep|^{\frac{5p-4}{p}} \,,
				\end{aligned}
			\end{equation*}
			whence \eqref{20220424_6} follows by taking $p=\frac 3 2$\,.
			}

	\medskip
	
	\textbf{Step 7: Final estimate.} 
	Let $\ffi\in\Cc^1(\Omega)$ with $\|\ffi\|_{C^{0,1}(\Omega)}\le 1$\,. 
		By definition of $Ju_\ep=\partial T_{u_\ep}$ we have 
		\begin{equation}\label{peru}
			\begin{aligned}
				\langle Ju_\ep,\ffi\rangle_{\Omega}=&\int_{\Omega}\nabla\ffi\cdot\ud\lambda_\ep\\
				=&\sum_{B\in\iteuno}\int_{B}\nabla\ffi\cdot\ud\lambda_\ep+\sum_{m=1}^{M_\ep}\int_{D^m\setminus \bigcup_{B\in\iteuno}B}\nabla\ffi\cdot\ud\lambda_\ep\\		
				&+\int_{\Omega\setminus(\bigcup_{m=1}^{M_\ep}D^m\cup\bigcup_{B\in\iteuno}B)}\nabla\ffi\cdot\ud\lambda_\ep\\
			\end{aligned}
		\end{equation}
		and, analogously, by integrating by parts
		\begin{equation}\label{perv}
			\begin{aligned}
				\langle{\pi} \widehat\mu_\ep,\ffi\rangle_{\Omega}=&\int_{\Omega}\nabla\ffi\cdot\ud\widehat\lambda_\ep\\
				=&\sum_{B\in\iteuno}\int_{B}\nabla\ffi\cdot\ud\widehat\lambda_\ep+\sum_{m=1}^{M_\ep}\int_{D^m\setminus \bigcup_{{B\in\iteuno}}B}\nabla\ffi\cdot\ud\widehat\lambda_\ep\\
				&+\int_{\Omega\setminus(\bigcup_{m=1}^{M_\ep}D^m\cup\bigcup_{B\in\iteuno}B)}\nabla\ffi\cdot\ud\widehat\lambda_\ep\,.
			\end{aligned}
		\end{equation}
		Using that $u_\ep, v_\ep\in H^1(\Omega\setminus(\bigcup_{m=1}^{M_\ep}D^m\cup\bigcup_{B\in\iteuno}B);\Ss^1)$\,, and hence $Ju_\ep=Jv_\ep=0$ in $\Omega\setminus(\bigcup_{m=1}^{M_\ep}D^m\cup\bigcup_{B\in\iteuno}B)$\,,
		we can integrate by parts the last integrals in \eqref{peru} and \eqref{perv}, thus obtaining
		\begin{equation}\label{20220424_1}
			\begin{aligned}
				\Big|\langle Ju_\ep-{\pi}\widehat\mu_\ep,\ffi\rangle_{\Omega}\Big|\le&|\lambda_\ep|\big(\bigcup_{B\in\iteuno}B\cup\bigcup_{m=1}^{M_\ep}D^m\big)+|\widehat\lambda_\ep|\big(\bigcup_{B\in\iteuno}B\cup\bigcup_{m=1}^{M_\ep}D^m\big)\\
				&{+\Big|\sum_{m=1}^{M_\ep}\int_{\partial D^m}\ffi\,(\lambda_\ep-\widehat\lambda_\ep)\cdot\nu\ud\Huno\Big|}\\
				&{+\Bigg|\sum_{\newatop{B\in\iteuno}{B\cap\bigcup_{m=1}^{M_\ep}D^m=\varnothing}}\int_{\partial B}\ffi\,(\lambda_\ep-\widehat\lambda_\ep)\cdot\nu\ud\Huno\Bigg|}\\
				\le& C\ep|\log\ep|^{{\frac 5 2}}+C\ep^{\frac 2 3}|\log\ep|^{\frac 7 3}\\
				&{+\Big|\sum_{m=1}^{M_\ep}\int_{\partial D^m}\ffi\,(\lambda_\ep-\widehat\lambda_\ep)\cdot\nu\ud\Huno\Big|}\\
				&{+\Bigg|\sum_{\newatop{B\in\iteuno}{B\cap\bigcup_{m=1}^{M_\ep}D^m=\varnothing}}\int_{\partial B}\ffi\,(\lambda_\ep-\widehat\lambda_\ep)\cdot\nu\ud\Huno\Bigg|}\,,
			\end{aligned}
		\end{equation}
		where in the last inequality we have used \eqref{20220424_5} and \eqref{20220424_6}.
		
		We now estimate the remaining integrals on the right-hand side of \eqref{20220424_1}.
			By \eqref{quasiquasi} and \eqref{stima_lambdahat}, using that $\deg(u_\ep,\partial D^m)=\deg(v_\ep,\partial D^m)$, for every $m=1,\ldots,M_\ep$, we have that 
		\begin{equation*}
			\begin{aligned}
				\Big|\int_{\partial D^m}\ffi(\lambda_\ep-\widehat\lambda_\ep)\cdot\nu\ud\Huno\Big|\le&\mathrm{osc}_{D^m}(\ffi)\Big(\int_{\partial D^m}|\lambda_\ep|\ud \mathcal H^1+\int_{\partial D^m}|\widehat\lambda_\ep|\ud\mathcal H^1\Big)\\
				\le&\mathrm{diam}(D^m)C|\log\ep|^{3}\,,
			\end{aligned}
		\end{equation*}
		which, summing over $m$\,, and using \eqref{sommadiametri}, yields
		\begin{equation}\label{20220424_3}
			\sum_{m=1}^{M_\ep}\Big|\int_{\partial D^m}\ffi(\lambda_\ep-\widehat\lambda_\ep)\cdot\nu\ud\Huno\Big|\le C\ep|\log\ep|^{5}\,.
		\end{equation}
		Analogously, using \eqref{quasiquasi} and \eqref{stima_lambdahat2} together with \eqref{sommaraggi}, we obtain
		\begin{equation}\label{20220424_4}
			\sum_{B\in\iteuno}\Big|\int_{\partial B}\ffi(\lambda_\ep-\widehat\lambda_\ep)\cdot\nu\ud\Huno\Big|\le C\ep|\log\ep|^{3}\,.
		\end{equation}
		Therefore, by \eqref{20220424_1}, \eqref{20220424_3}, and \eqref{20220424_4}, for $\ep$ small enough, we get
		\begin{equation*}
			\Big|\langle Ju_\ep-{\pi}\widehat\mu_\ep,\ffi\rangle_{\Omega}\Big|\le C\ep|\log\ep|^{5}\,,
		\end{equation*}
		for some constant $C$ independent of $\ffi$\,.
	Taking the supremum on $\varphi\in \Cc^{1}(\Omega)$ such that $\|\varphi\|_{C^{0,1}}\leq 1$ we infer that $	\|Ju_\ep-{\pi}\widehat\mu_\ep\|_{\flt,\Omega}\leq C\ep|\log\ep|^5$\,, whence
	\eqref{20220424_7} follows. This concludes the proof of (i).

	\vskip3mm
	{\it Proof of (ii).} We can assume without loss of generality that $\{u_\ep\}_{\ep}$ satisfies \eqref{enbound} for some $C>0$\,.
	By \eqref{nuovobound1} and \eqref{sumradii}, applying Theorem \ref{alipons}(ii)\,, we get that
	\begin{equation*}
		\liminf_{\ep\to 0}\frac{\F_\ep(u_\ep)}{|\log\ep|}\ge\liminf_{\ep\to 0}\frac{F(\B_\ep,\mu_\ep,\Omega)}{|\log\rad(\B_\ep)|}\ge \pi|\mu|(\Omega)\,,
	\end{equation*}
	i.e., the claim.
	\vskip3mm
	{\it Proof of (iii).} 
	We divide the proof into three steps. In the first one we construct the sequence $\{u_\ep\}_\ep$ and we show that it satisfies \eqref{eq:limsup}; whereas the second and third steps are devoted to show that $\|Ju_\ep-\pi\mu\|_{\flt,\Omega}\to 0$ (as $\ep\to 0$).
	
	\medskip
	\textbf{Step 1 (Definition of $u_\ep$ and energy estimates):}
	Let $\mu=\sum_{i=1}^{I}z^i\delta_{x^i}$ with $z^i\in\Z\setminus\{0\}$ and $x^i\in\Omega$ for every $i=1,\ldots,I$\,. By standard density arguments in $\Gamma$-convergence, we can assume that $|z^i|=1$ for every $i=1,\ldots,I$\,.
	
	Let $R>0$ be such that the balls $B_{2R}(x^i)$ are pairwise (essentially) disjoint and contained in $\Omega$\,. For every $\rho>0$ we set $\Omega_{\rho}(\mu):=\Omega\setminus\bigcup_{i=1}^I\overline B_{\rho}(x^i)$ and we define $\Theta:\Omega_{R}(\mu)\to \R$ as $\Theta(\cdot):=\sum_{i=1}^{I}z^i\vartheta(\cdot-x^i)$\,,
	where $\vartheta$ is the function defined in \eqref{deftheta0}.

	
	For every $0<\rho_1<\rho_2$ and for every $x\in\R^2$ we set $A_{\rho_1,\rho_2}(x):=B_{\rho_2}(x)\setminus \overline B_{\rho_1}(x)$\,.
	Let moreover $\sigma_R\in C^\infty(A_{R,2R}(0);[0,1])$ be such that $\sigma_{R}\equiv 1$ in $A_{\frac 7 4R,2R}(0)$ and $\sigma_R\equiv 0$ in $A_{R,\frac 5 4 R}(0)$\,. Analogously, for every $\ep>0$\,, let $\sigma_\ep\in C^\infty(B_\ep(0);[0,1])$ be such that $\sigma_\ep\equiv 0$ in $B_{\frac \ep 4}(0)$ and $\sigma_\ep\equiv 1$ in $A_{\frac 3 4\ep,\ep}(0)$\,. {We can also assume that
	\begin{equation}\label{stimacutoff}
	|\nabla\sigma_\ep(x)|\le \frac{C}{\ep}\quad\textrm{for every  }x\in B_\ep(0)\,,\qquad|\nabla\sigma_R(x)|\le \frac{C}{R}\quad\textrm{for every  }x\in A_{R,\frac 5 4 R}(0)\,,
	\end{equation} 
	for some constant $C>0$ independent of $\ep$ (and of $x$)\,.
	} 
	
	For $\ep<R$ we define the function $\vartheta_\ep:\Omega\to\R$ as
	\begin{equation*}
		\vartheta_\ep(x):=\left\{\begin{array}{ll}
			\sigma_\ep(x-x^i)z^i\vartheta(x-x^i) &\textrm{if }x\in B_\ep(x^i)\setminus\{x^i\}\textrm{ for some }i\\
			z^i\vartheta(x-x^i)&\textrm{if }x\in A_{\ep,R}(x^i)\textrm{ for some }i\\
			(1-\sigma_R(x-x^i))z^i\vartheta(x-x^i)+\sigma_{R}(x-x^i)\Theta(x)&\textrm{if }x\in A_{R,2R}(x^i)\textrm{ for some }i\\
			\Theta(x)&\textrm{if } x\in\Omega_{2R}(\mu)
		\end{array}
		\right.
	\end{equation*}
	and we define $u_\ep:\Omega\to \Ss^1$ as 
	\begin{equation}\label{recovery}
	u_\ep(\cdot):=e^{\imath\vartheta_\ep(\cdot)}\,.
	\end{equation}
	Then, $S_{u_\ep}=\bigcup_{i=1}^I S^i_\ep$ where
	\begin{equation}\label{saltini}
		S_\ep^i:=\Big\{(x^i_1,x_2)\,:\,x^i_2-\frac 3 4\ep<x_2<x_2^i-\frac \ep 4\Big\}\qquad\textrm{ for every }i=1,\ldots,I\,,
	\end{equation}
	and hence 
	\begin{equation}\label{uno}
		\Huno(\overline S_{u_\ep})=I\frac \ep 2=|\mu|(\Omega)\frac \ep 2\,.
	\end{equation}
In particular, the energy term 
\begin{equation}\label{unopoi}
\frac{\mathcal H^1(\overline S_{u_\ep})}{\ep|\log\ep|}\rightarrow0\qquad \text{as } \ep\rightarrow0^+\,.
\end{equation}
 Let us analyse the stored elastic energy.
	It is easy to see that
	\begin{equation}\label{unico}
		\frac{1}{2}\int_{\Omega_{2R}(\mu)}|\nabla u_\ep|^2\ud x{\le}\frac{1}{2}\int_{\Omega_{2R}(\mu)}|\nabla \Theta|^2\ud x\le C(I)\log\frac{\mathrm{diam}(\Omega)}{2R}\,,
	\end{equation}
where $C(I)>0$ is a constant depending only on $I$\,.
Similarly, by the second estimate in \eqref{stimacutoff}, we deduce 
\begin{equation}\label{unicodue}
\begin{aligned}
\frac 1 2\int_{A_{R,2R}(x^i)}|\nabla u_\ep|^2\ud x\le&\int_{A_{R,2R}(x^i)}|\nabla\vartheta(x-x^i)|^2\ud x\\
		&+\int_{A_{R,2R}(x^i)}\bigg|\nabla\bigg(\sigma_R(x-x^i)\Big(\sum_{{j\neq i}}z^j\vartheta(x-x^j)\Big)\bigg)\bigg|^2\ud x\\
		\le& C(I,R)\,,
\end{aligned}
\end{equation}
where we have estimated 
\begin{align*}
	&\int_{A_{R,2R}(x^i)}\bigg|\nabla\bigg(\sigma_R(x-x^i)\Big(\sum_{{j\neq i}}z^j\vartheta(x-x^j)\Big)\bigg)\bigg|^2\ud x\\
	&\qquad\leq C(I)+\int_{A_{R,2R}(x^i)}\big|\sum_{{j\neq i}}z^j\nabla\vartheta(x-x^j)\big|^2\ud x\leq C(I,R).
\end{align*}
Furthermore, {by the first estimate in \eqref{stimacutoff}, for every $i=1,\ldots,I$, and for every $x\in B_{R}(x^i)$ it holds
{\begin{align}\label{unicotre}
|\nabla u_\ep(x)|=|\nabla\vartheta_\ep(x)|&\le2\pi \Big|\nabla\sigma_\ep(x-x^i)\chi_{B_\ep(x^i)}(x)\Big|+\Big|\nabla\vartheta(x-x^i)\chi_{A_{\frac\ep 4,R}(x^i)}(x)\Big|
\nonumber\\
&\le\frac{C}{\ep}\chi_{B_\ep(x^i)}(x)+\frac{1}{|x-x^i|}\chi_{A_{\frac\ep 4,R}(x^i)}(x)\,,
\end{align}}
for some constant $C$ independent of $\ep$ (and of $x$).
}
It follows that, for every $i=1,\ldots,I$,
	\begin{eqnarray}\label{stimadelimperatore}
		\frac 1 2\int_{B_\ep(x^i)}|\nabla u_\ep|^2\ud x&
		\le& C\,,\\
		\frac 1 2\int_{A_{\ep,R}(x^i)}|\nabla u_\ep|^2\ud x&\leq&\pi\log\frac{R}{\ep}\,,\nonumber
	\end{eqnarray}
	which, together with \eqref{unico} and \eqref{unicodue}, implies that
	\begin{equation*}\label{due}
		\limsup_{\ep\to 0}\frac{1}{2|\log\ep|}\int_{\Omega}|\nabla u_\ep|^2\ud x\leq\pi I=\pi|\mu|(\Omega)\,,
	\end{equation*}
	which, together with \eqref{unopoi}, yields \eqref{eq:limsup}.
	
Now, in order to conclude the proof of (iii) of Theorem \ref{mainthm}, it remains to prove that
\begin{align}\label{conclusionflat}
	\|Ju_\ep-\pi\mu\|_{\textrm{flat},\Om}\rightarrow0\qquad \text{as }\ep\to0.
	\end{align}
To this purpose in view of Remark \ref{flat_rmk}, it is enough to prove the following two facts: There exists a constant $C>0$ (independent of $\ep$) such that
\begin{equation}\label{quasitre}
	|Ju_\ep|(\Omega)\le C\,,
\end{equation}
and 
\begin{align}\label{strict_conclusion}
u_\ep\strictly \bar u\qquad\text{ in }BV(\Om;\R^2)\,,
\end{align}
for some $\bar u\in W^{1,1}(\Om;\Ss^1)$ with $\textrm{Det}(\nabla \bar u)=\pi\mu$.
%
%
%
%
%
\medskip

\textbf{Step 2: Proof of \eqref{quasitre}.}
In order to show \eqref{quasitre}, thanks to Remark \ref{mucci_rmk}, it is sufficient to show that for all $\ep$ small enough there exists a sequence of maps $\{v_{\ep,k}\}_{k\in\N}\subset C^1(\Om;\R^2)\cap BV(\Omega;\R^2)$ such that 
\begin{equation}\label{quasitre2}
\sup_{k\in\N}|Jv_{\ep,k}|(\Om)\le C\,,
\end{equation}
for some $C>0$ independent of $k$ and $\ep$\,, and
\begin{equation}\label{strict_conclusion2}
v_{\ep,k}\strictly u_\ep \quad\textrm{in }BV(\Omega;\R^2)\qquad\textrm{(as $k\to +\infty$)}\,.
\end{equation}
Furthermore, by standard density arguments, it is enough to construct $\{v_{\ep,k}\}_{k}\subset W^{1,\infty}(\Om;\R^2)$\,.
%
	Let $0<\ep\ll R$ be fixed. 
For every $i=1,\ldots,I$\,, $k\in\N$\,, $ t_1<t_2$ we set $S^i_{\ep,k,t_1,t_2}:=[x_1^i-\frac{\ep}{k}, x_1^i+\frac{\ep}{k}]\times [x_2^i+t_1, x_2^i+t_2]$\,.
	We will modify $u_\ep$ in the set $S^i_{\ep,k,-\ep-\frac\ep k,\frac\ep k}$, that is a neighborhood of the jump set $S^i_\ep$ in \eqref{saltini}. Precisely, for every $i=1,\ldots,I$, and for every $k\in \N$, let $u^i_{\ep,k,\mathrm{int}}\in C^\infty(S^i_{\ep,k,-\ep-\frac\ep k,\frac\ep k};\R^2)$ be defined as
	\begin{eqnarray*}
		u^i_{\ep,k,\mathrm{int}}(x)&:=&\Big(\frac12-\frac{k}{2\ep}(x_1-x_1^i)\Big)u_\ep\Big(x_1^i-\frac{\ep}{k};x_2\Big)\\
		&&+\Big(\frac12+\frac{k}{2\ep}(x_1-x_1^i)\Big)u_\ep\Big(x_1^i+\frac{\ep}{k};x_2\Big).
	\end{eqnarray*}
	Moreover, for every $i=1,\ldots,I$, and for every $k\in\N$ we denote by  $u^i_{\ep,k,\mathrm{bottom}}\in C^\infty(S^i_{\ep,k,-\ep-\frac\ep k,-\ep};\R^2)$ and $u^i_{\ep,k,\mathrm{up}}(S^i_{\ep,k,0,\frac\ep k};\R^2)$ the functions
	\begin{eqnarray*}
		u^i_{\ep,k,\mathrm{bottom}}(x)&:=&\frac{k}{\ep}(x_2^i-\ep-x_2)u_\ep(x)+\Big(1-\frac{k}{\ep}(x_2^i-\ep-x_2)\Big)u^i_{\ep,k,\mathrm{int}}(x)\,,\\
			u^i_{\ep,k,\mathrm{up}}(x)&:=&\frac{k}{\ep}(x_2-x_2^i)u_\ep(x)+\Big(1-\frac{k}{\ep}(x_2-x_2^i)\Big)u^i_{\ep,k,\mathrm{int}}(x)\,.
	\end{eqnarray*}
Finally we define $v_{\ep,k}:\Omega\to\R^2$ as 
\begin{equation*}
	v_{\ep,k}(x):=\left\{\begin{array}{ll}
		u^i_{\ep,k,\mathrm{int}}(x)
	&\textrm{if }x\in S^i_{\ep,k,-\ep,0}\textrm{ for some }i\\
		u_{\ep,k,\mathrm{up}}^i(x)&\textrm{if }x\in S^i_{\ep,k,0,\frac{\ep}{k}}\textrm{ for some }i\\
		u_{\ep,k,\mathrm{bottom}}^i(x)&\textrm{if }x\in S^i_{\ep,k,-\ep-\frac{\ep}{k},-\ep}\textrm{ for some }i\\\
		u_\ep(x)&\textrm{elsewhere in }\Omega\,.
	\end{array}
	\right.
\end{equation*}
An easy check shows that the function $v_{\ep,k}:\Omega\to\R^2$ is Lipschitz continuous in $\Om$. 
Indeed, $u^i_{\ep,k,\mathrm{int}}$ is Lipschitz continuous in $S^i_{\ep,k,0,\frac\ep k}$ (and smooth in its interior), and $u^i_{\ep,k,\mathrm{bottom}}$ and $u^i_{\ep,k,\mathrm{up}}$ are Lipschitz as well in $S^i_{\ep,k,-\ep-\frac{\ep}{k},-\ep}$ and  $S^i_{\ep,k,0,\frac{\ep}{k}}$ (respectively). 

We now check that $\{v_{\ep,k}\}_{k\in\N}$ satisfies \eqref{strict_conclusion2}.
To this purpose it is enough to show that 
\begin{align*}
\lim_{k\to +\infty}\int_{S^i_{\ep,k,-\ep-\frac\ep k,\frac\ep k}}|\nabla v_{\ep,k}|\ud x=\int_{S^i_\ep}|[u_\ep]|\ud\mathcal H^1,
\end{align*}
for all $i=1,\dots,I$.
Furthermore, since $v_{\ep,k}\to u_\ep$ pointwise a.e. in $\Om$ (as $k\to +\infty$), it is sufficient to prove that 
\begin{align}\label{stimaperlimsup}
	\limsup_{k\rightarrow+\infty}\int_{S^i_{\ep,k,-\ep-\frac\ep k,\frac\ep k}}|\nabla v_{\ep,k}|\ud x\leq \int_{S^i_\ep}|[u_\ep]|\ud\mathcal H^1\,,
\end{align} 
for all $i=1,\dots,I$. We notice that
\begin{align*}
	&\frac{\partial u^i_{\ep,k,\textrm{int}}}{\partial x_1}(x)=\frac{k}{2\ep}u_\ep\Big(x_1^i-\frac{\ep}{k};x_2\Big)-\frac{k}{2\ep}u_\ep\Big(x_1^i+\frac{\ep}{k};x_2\Big)\,,\\
	&\frac{\partial u^i_{\ep,k,\textrm{int}}}{\partial x_2}(x)=\Big(\frac12-\frac{k}{2\ep}(x_1-x_1^i)\Big)\frac{\partial u_\ep}{\partial x_2}\Big(x_1^i-\frac{\ep}{k};x_2\Big)+\Big(\frac12+\frac{k}{2\ep}(x_1-x_1^i)\Big)\frac{\partial u_\ep}{\partial x_2}\Big(x_1^i+\frac{\ep}{k};x_2\Big)\,,
\end{align*}
	whence, using that
\begin{equation}\label{costru}
|x_1-x_1^i|\leq\frac\ep k\textrm{ in }S^i_{\ep,k,-\ep-\frac\ep k,\frac \ep k}\,,
\end{equation}	
we deduce
\begin{equation}\label{cisiamo}
\begin{aligned}
\int_{S^i_{\ep,k,-\ep,0}}|\nabla v_{\ep,k}|\ud x=&\int_{S^i_{\ep,k,-\ep,0}}|\nabla u^i_{\ep,k,\textrm{int}}|\ud x\\
\leq& \int_{S^i_{\ep,k,-\ep,0}}\Big|\frac{\partial u^i_{\ep,k,\textrm{int}}}{\partial x_1}\Big|\ud x+\int_{S^i_{\ep,k,-\ep,0}}\Big|\frac{\partial u^i_{\ep,k,\textrm{int}}}{\partial x_2}\Big|\ud x\\
\leq& \int_{-\ep}^0\Big|u_\ep\Big(x_1^i+\frac{\ep}{k};x_2\Big)-u_\ep\Big(x_1^i-\frac{\ep}{k};x_2\Big) \Big|\ud x_2\\
&+\int_{-\ep}^0\int_{-\frac\ep k}^{\frac\ep k}\Big|\frac{\partial u_\ep}{\partial x_2}\Big(x_1^i-\frac{\ep}{k};x_2\Big)\Big|+\Big|\frac{\partial u_\ep}{\partial x_2}\Big(x_1^i+\frac{\ep}{k};x_2\Big)\Big|\ud x_1\ud x_2\\
=&\int_{S^i_\ep}|[u_\ep]|\ud\mathcal H^1+\mathrm{o}_k(1)\,,
\end{aligned}
\end{equation}
where $\mathrm{o}_k(1)$ tends to $0$ as $k\rightarrow+\infty$\,. 
Furthermore,
	\begin{align*}
	&\frac{\partial	u^i_{\ep,k,\mathrm{bottom}}}{\partial x_1}(x)=\frac{k}{\ep}(x_2^i-\ep-x_2)\frac{\partial u_\ep}{\partial x_1}(x)+\Big(1-\frac{k}{\ep}(x_2^i-\ep-x_2)\Big)\frac{\partial u^i_{\ep,k,\mathrm{int}}}{\partial x_1}(x)\,,\\
	&\frac{\partial	u^i_{\ep,k,\mathrm{bottom}}}{\partial x_2}(x)=-\frac{k}{\ep}u_\ep(x)+\frac{k}{\ep}u^i_{\ep,k,\mathrm{int}}(x)+\frac{k}{\ep}(x_2^i-\ep-x_2)\frac{\partial u_\ep}{\partial x_2}(x)\\
	&\phantom{\frac{\partial	u^i_{\ep,k,\mathrm{bottom}}}{\partial x_2}(x)=}+\Big(1-\frac{k}{\ep}(x_2^i-\ep-x_2)\Big)\frac{\partial u^i_{\ep,k,\mathrm{int}}}{\partial x_2}(x)\,,\\
	&\frac{\partial	u^i_{\ep,k,\mathrm{up}}}{\partial x_1}(x)=\frac{k}{\ep}(x_2-x_2^i)\frac{\partial u_\ep}{\partial x_1}(x)+\Big(1-\frac{k}{\ep}(x_2-x_2^i)\Big)\frac{\partial u^i_{\ep,k,\mathrm{int}}}{\partial x_1}(x)\,,\\
	&\frac{\partial	u^i_{\ep,k,\mathrm{up}}}{\partial x_2}(x)=\frac{k}{\ep}u_\ep(x)-\frac{k}{\ep}u^i_{\ep,k,\mathrm{int}}(x)+\frac{k}{\ep}(x_2-x_2^i)\frac{\partial u_\ep}{\partial x_2}(x)+\Big(1-\frac{k}{\ep}(x_2-x_2^i)\Big)\frac{\partial u^i_{\ep,k,\mathrm{int}}}{\partial x_2}(x)\,;
\end{align*}
hence, using that {$|u_\ep|,|u^i_{\ep,k,\textrm{int}}|\le 1$\,,}
 \eqref{costru} and the fact that
\begin{equation*}
 |\ep+x_2-x_2^i|\leq\frac\ep k\textrm{ in }S^i_{\ep,k,-\ep-\frac\ep k,-\ep}\,,\quad
|x_2-x_2^i|\leq\frac\ep k\textrm{ in } S^i_{\ep,k,0,\frac\ep k},
\end{equation*}
we can estimate
{
\begin{equation}\label{estimate_sup}
\begin{aligned}
&\Big|\frac{\partial u^i_{\ep,k,\textrm{int}}}{\partial x_1}(x)\Big|\le \frac{k}{\ep}\,,\\
&\Big|\frac{\partial u^i_{\ep,k,\textrm{int}}}{\partial x_2}(x)\Big|\leq \Big|\nabla u_\ep \Big(x^i_1-\frac{\ep}{k};x_2\Big)\Big|+\Big|\nabla u_\ep \Big(x^i_1+\frac{\ep}{k};x_2\Big)\Big|\,,\\
&\Big|\frac{\partial	u^i_{\ep,k,\mathrm{bottom}}}{\partial x_1}(x)\Big|\leq |\nabla u_\ep(x)|+\frac{k}{\ep}\,,\\
&\Big|\frac{\partial	u^i_{\ep,k,\mathrm{bottom}}}{\partial x_2}(x)\Big|\leq 2\frac k \ep+|\nabla u_\ep(x)|+\Big|\nabla u_\ep \Big(x^i_1-\frac{\ep}{k};x_2\Big)\Big|+\Big|\nabla u_\ep \Big(x^i_1+\frac{\ep}{k};x_2\Big)\Big|\,,\\
& \Big|\frac{\partial	u^i_{\ep,k,\mathrm{up}}}{\partial x_1}(x)\Big|\leq  |\nabla u_\ep(x)|+\frac{k}{\ep}\,,\\
& \Big|\frac{\partial	u^i_{\ep,k,\mathrm{up}}}{\partial x_2}(x)\Big|\leq 2\frac k \ep+|\nabla u_\ep(x)|+\Big|\nabla u_\ep \Big(x^i_1-\frac{\ep}{k};x_2\Big)\Big|+\Big|\nabla u_\ep \Big(x^i_1+\frac{\ep}{k};x_2\Big)\Big|\,.
\end{aligned}
\end{equation}
}
By \eqref{estimate_sup} and {\eqref{stimadelimperatore}} it follows, in particular, that the integrals of $|\nabla v_{\ep,k}|$ on  $S^i_{\ep,k,-\ep,-\ep-\frac\ep k}$ and $S^i_{\ep,k,0,\frac\ep k}$ are both negligible as $k\rightarrow\infty$\,. This fact, together with \eqref{cisiamo}, implies \eqref{stimaperlimsup} and, in turn, \eqref{strict_conclusion2}.

In order to prove \eqref{quasitre2}, we notice that,
in view of \eqref{estimate_sup} {and of \eqref{unicotre},}
%
it holds
\begin{equation}\label{stimaprima}
\begin{aligned}
	\int_{S^i_{\ep,k,-\ep,0}}|J v_{\ep,k}|\ud x=&	\int_{S^i_{\ep,k,-\ep,0}}|Ju^i_{\ep,k,\textrm{int}}|\ud x\\
	\le&\int_{S^i_{\ep,k,-\ep,0}}\Big|\frac{\partial u^i_{\ep,k,\textrm{int}}}{\partial x_1}(x)\Big| \Big|\frac{\partial u^i_{\ep,k,\textrm{int}}}{\partial x_2}(x)\Big|\ud x\\
	\le& 2\int_{-\ep}^{0} \Big|\nabla u_\ep \Big(x^i_1-\frac{\ep}{k};x_2\Big)\Big|+\Big|\nabla u_\ep \Big(x^i_1+\frac{\ep}{k};x_2\Big)\Big|\ud x_2\\
	\le&C\,,
	\end{aligned}
	\end{equation}
	for some constant $C$ independent of $\ep$ and $k$\,.
	
	Analogously, using again \eqref{estimate_sup} and \eqref{stimadelimperatore}, one can prove that there exists a universal constant $C>0$ such that
	\begin{equation}\label{stimaseconda}
	\begin{aligned}
	\int_{S^i_{\ep,k,-\ep-\frac\ep k,-\ep}}|J v_{\ep,k}|\ud x=&	\int_{S^i_{\ep,k,-\ep-\frac\ep k,-\ep}}|Ju^i_{\ep,k,\textrm{bottom}}|\ud x
	\le C\,,\\
	\int_{S^i_{\ep,k,0,\frac\ep k}}|J v_{\ep,k}|\ud x=&	\int_{S^i_{\ep,k,0,\frac\ep k}}|Ju^i_{\ep,k,\textrm{up}}|\ud x
	\le C\,.
\end{aligned}
\end{equation}
By \eqref{stimaprima} and \eqref{stimaseconda}, the estimate in \eqref{quasitre2} follows.
				\medskip
				
				\textbf{Step 3: Proof of \eqref{strict_conclusion}.} 
			It is easy to see that the functions $\vartheta_\ep$ tend pointwise to $\bar\vartheta$ defined as	
					\begin{equation*}
					\bar\vartheta(x):=\left\{\begin{array}{ll}
						z^i\vartheta(x-x^i)&\textrm{if }x\in B_{R}(x^i)\textrm{ for some }i\\
						(1-\sigma_R(x-x^i))z^i\vartheta(x-x^i)+\sigma_{R}(x-x^i)\Theta(x)&\textrm{if }x\in A_{R,2R}(x^i)\textrm{ for some }i\\
						\Theta(x)&\textrm{if } x\in\Omega_{2R}(\mu).
					\end{array}
					\right.
				\end{equation*}
				By definition, the function $\bar u$ defined by $\bar u(x):=e^{\imath\bar\vartheta(x)}$ belongs to $W^{1,1}(\Om;\Ss^1)$ and satisfies $\textrm{Det}(\nabla \bar u)=\pi\mu$\,.
				Notice moreover that, since $\bar\vartheta-\vartheta_\ep\neq0$ only on $\bigcup_{i=1}^IB_\ep(x^i)$ and $\bar u\in W^{1,1}(\Om;\Ss^1)$, we can prove that $u_\ep\strictly \bar u$ in $BV(\Omega;\R^2)$ if we show that 
				\begin{equation}\label{verofin}
				|\mad u_\ep|(\cup_{i=1}^I\overline B_\ep(x^i))\rightarrow0\qquad\textrm{ as }\ep\to 0\,.
				\end{equation}
				But, by \eqref{stimadelimperatore} and by H\"older inequality
				$$
				\int_{B_\ep(x^i)}|\nabla u_\ep|\ud x\leq C\ep\Big(\int_{B_\ep(x^i)}|\nabla u_\ep|^2\Big)^{\frac12}\ud x\leq C\ep\qquad \forall i=1,\dots,I\,,
				$$
				which together with \eqref{uno} implies \eqref{verofin}.
				This concludes the proof of \eqref{strict_conclusion}, and in turn of \eqref{conclusionflat} and of the whole Theorem \ref{mainthm}.
			\end{proof}
\subsection*{Dirichlet boundary conditions}
We conclude this section dealing with prescribed boundary conditions on $\partial\Omega$\,.
To this purpose, let $\widetilde{\Omega},\widehat \Om\subset\R^2$ be open bounded sets with {$\widetilde\Omega\subset\subset\Om\subset\subset\widehat \Om$}\,, and let $w\in H^1(\widehat\Om\setminus\widetilde\Omega;\Ss^1)$\,. Admissible functions for the problem with boundary conditions will be given by 
\begin{equation*}
	\mathcal {AD}^w(\Om):=\{u\in SBV^2(\widehat \Om;\Ss^1):u=w\text{ on }\widehat\Om\setminus \overline\Om\}.
\end{equation*}
For every $\ep>0$\,, let $\F^w_\ep:L^2(\Omega;\R^2)\to [0,\infty]$ be the functional defined by
\begin{equation*}
	\F^w_\ep(u):=\left\{\begin{array}{ll}
		\displaystyle \frac 1 2 \int_{\Omega}|\nabla u|^2\ud x+\frac 1 \ep\Huno(\overline{S}_u)&\textrm{if }u\in \mathcal{AD}^w(\Om)\\
		\displaystyle +\infty&\textrm{elsewhere in }L^2(\Omega;\R^2)\,.
	\end{array}
	\right.
\end{equation*}
Notice that we can have $S_u\cap \partial\Om\neq\varnothing$, and that we take the closure $\overline S_u$ of $S_u$ in $\widehat\Om$ (in particular, the functional $\mathcal F_\ep$ penalizes jumps of $u$ also on $\partial\Om$). 

We show that, up to slight modifications in the proof of Theorem \ref{mainthm}, we can obtain a compactness and $\Gamma$-convergence result for the functional $\F_\ep^w$\,.
Specifically, we will prove the following statement.
\begin{theorem}\label{mainthmW}
	The following $\Gamma$-convergence result holds true.
	\begin{itemize}
		\item[(i)] (Compactness) Let $\{u_\ep\}_{\ep}\subset L^2(\Omega;\R^2)$ be such that 
		\begin{equation}\label{enboundW}
			\sup_{\ep>0}\frac{\F^w_\ep(u_\ep)}{|\log\ep|}\le C,
		\end{equation}
		for some $C>0$\,. Then there exists $\mu\in X(\overline\Omega)$\,, such that , up to a subsequence, $\|Ju_\ep-\pi\mu\|_{\flt,U}\to 0$ (as $\ep\to 0$)\,, for any open set $U$ with $\Omega\subset\subset U\subset\subset\widehat\Omega$.
		{Moreover, it holds 
			\begin{equation}\label{grado_mu}
			\mu(U)=\deg(w,\partial \Om)\,,
			\end{equation} 
		for any open set $U$ with $\Omega\subset\subset U\subset\subset\widehat\Omega$.}
		\item[(ii)] ($\Gamma$-liminf inequality) For every $\mu\in X(\overline\Omega)$ and for every $\{u_\ep\}_{\ep}\subset L^2(\Omega;\R^2)$ such that $\|Ju_\ep-\pi\mu\|_{\flt,U}\to 0$ (as $\ep\to 0$) for any open set $U$ with $\Omega\subset\subset U\subset\subset\widehat\Omega$, it holds
		\begin{equation*}
			\pi|\mu|(\overline\Omega)\le \liminf_{\ep\to 0}\frac{\F^w_\ep(u_\ep)}{|\log\ep|}\,.
		\end{equation*}
		\item[(iii)] ($\Gamma$-limsup inequality) For every $\mu\in X(\overline\Omega)$ satisfying \eqref{grado_mu}, there exists  $\{u_\ep\}_{\ep}\subset L^2(\Omega;\R^2)$ with $\|Ju_\ep-\pi\mu\|_{\flt,U}\to 0$ (as $\ep\to 0$) for any open set $U$ with $\Omega\subset\subset U\subset\subset\widehat\Omega$, such that
		\begin{equation*}
			\pi|\mu|(\overline\Omega)\geq\limsup_{\ep\to 0}\frac{\F^w_\ep(u_\ep)}{|\log\ep|}\,.
		\end{equation*}
	\end{itemize}
\end{theorem}
 \begin{proof}
We divide the proof into three steps corresponding to each of the items in the statement.

\medskip
{\bf Step 1: Proof of (i).} 
Notice that, by \eqref{enboundW}, we can assume without loss of generality that $u_\ep\in \mathcal{AD}(\widehat\Omega, \Om)$ for all $\ep>0$ (see \eqref{admissiblebdrysenzaw}).
 Moreover, recalling the definition of $\F_\ep$ in \eqref{defEne}, by \eqref{enboundW} we have
 \begin{equation}\label{ancheliminf}
 \F_\ep(u_\ep; \widehat\Omega)\le \F_\ep^w(u_\ep)+\|w\|^2_{H^1(\widehat\Omega\setminus\overline{\Omega};\R^2)}\le C|\log\ep|\,.
 \end{equation}
 Hence, we can apply Theorem \ref{mainthm}(i) to deduce that, up to a subsequence, $\|Ju_\ep-\pi\widehat\mu\|_{\flt,\widehat\Omega}\to 0$ for some $\widehat\mu\in X(\widehat\Omega)$\,. {Therefore, since, by construction, 
 $\supp Ju_\ep\subset\overline\Omega$\,, then $\supp \widehat \mu\subset\overline\Om$, and therefore $\|J u_\ep-\pi\mu\|_{\flt,U}\to 0$\,,
 for every open set $U$ with $\Omega\subset\subset U\subset\subset\widehat\Omega$\,.}
Finally, if $U$ is an open set with  $\Omega\subset\subset U\subset\subset\widehat\Omega$\,, we choose a function $\varphi\in \Cc^\infty(\widehat\Om)$ which equals $1$ on $U$, and write 
\begin{align*}
	\pi\mu(U)&=\langle \pi\mu,\varphi\rangle_{\widehat\Om}=\lim_{\ep\to 0}\langle Ju_\ep,\varphi\rangle_{\widehat\Om}=\lim_{\ep\to 0}\int_{\widehat\Om}\nabla \varphi\cdot\ud\lambda_{u_\ep}  =\lim_{\ep\to 0}\int_{\widehat\Om\setminus \overline U}\nabla \varphi\cdot\ud\lambda_{u_\ep}\nonumber\\
	&=\lim_{\ep\to 0}\int_{\partial U}\lambda_{u_\ep}\cdot \nu \ud\mathcal H^1=\lim_{\ep\to 0}\pi\deg(u_\ep,\partial U)=\pi\deg(w,\partial U)\,,
	\end{align*}
where we have used \eqref{def_grado} and the fact that $u_\ep\in \mathcal{AD}^w(\widehat\Om,\Om)$.
 \medskip
 
 {\bf Step 2: Proof of (ii).} We can assume without loss of generality that $\{u_\ep\}_\ep$ satisfies \eqref{enboundW}. By \eqref{ancheliminf} and by Theorem \ref{mainthm}(ii), we obtain
 \begin{equation*}
 \liminf_{\ep\to 0}\frac{\F_\ep^w(u_\ep)}{|\log\ep|}\ge \liminf_{\ep\to 0}\frac{\F_\ep(u_\ep;\widehat\Omega)}{|\log\ep|}\ge \pi|\mu|(\overline\Omega)\,,
 \end{equation*}
 i.e., the claim.
 
 \medskip
 {\bf Step 3: Proof of (iii).} {Since every $\mu\in X(\overline\Om)$ can be approximated by measures in $X(\Om)$ with respect to the flat distance, by standard density arguments in $\Gamma$-convergence, it is enough to construct the recovery sequence only for measures $\mu\in X(\Om)$.
 	
 	 Let $\widetilde\Omega\subset\subset\Omega'\subset\subset\Omega$ be an open set such that $\supp\mu\subset\Omega'$\,.
 Let moreover $\eta\in C^\infty(\Omega\setminus\Omega')$ be a cut-off function with $\eta\equiv 0$ in a neighborhood of $\partial\Omega'$ and $\eta\equiv 1$ in a neighborhood of $\partial\Omega$\,. 
 Let furthermore $\{\widetilde u_\ep\}_\ep=\{e^{\imath \widetilde\vartheta_\ep}\}_\ep$ indicate the sequence provided by \eqref{recovery}\,. 
 Then there exists a function $\vartheta^w\in SBV^2(\widehat\Omega\setminus\widetilde\Omega)$ such that $w=e^{\imath\vartheta^w}$ on $\widehat\Om\setminus \widetilde\Om$\,,  $S_{\vartheta^w}=S_{\widetilde\vartheta_\ep}$ in $\Om\setminus \overline{\Om'}$\,, and  $[\vartheta^w]=[\widetilde\vartheta_\ep]$ in $\Om\setminus \overline{\Om'}$\,. The last property can be achieved thanks to the fact that $\mu$ satisfies condition \eqref{grado_mu}.

 For every $\ep>0$ we define the function $u_\ep:\widehat\Omega\to\Ss^1$ as $u_\ep(\cdot):=e^{\imath\vartheta_\ep(\cdot)}$\,, where the lifting $\vartheta_\ep$ is defined by
 \begin{equation*}
 \vartheta_\ep(x):=\left\{\begin{array}{ll}
 \widetilde\vartheta_\ep(x)&\textrm{if }x\in \Omega'\\
 (1-\eta(x))\widetilde\vartheta_\ep(x)+\eta(x) \vartheta^w(x)&\textrm{if }x\in\Omega\setminus \Omega'\\
 \vartheta^w(x)&\textrm{if }x\in\widehat\Omega\setminus \Omega\,.
 \end{array}
 \right.
 \end{equation*}
 {It is easy to check that the sequence $\{u_\ep\}_\ep$ satisfies the desired properties.}}
 \end{proof}




\end{document}